\documentclass[a4paper]{scrartcl}%
\usepackage[utf8]{inputenc}%
\usepackage{amsthm,amsmath,amsfonts,amssymb}%
\numberwithin{equation}{section}
\usepackage{enumitem}%
\usepackage[backend=bibtex,citestyle=alphabetic,bibstyle=alphabetic,giveninits=true,maxbibnames=99,doi=true,isbn=false,url=false,eprint=false]{biblatex}
\usepackage{tikz-cd}%
\usepackage{mathtools}
\mathtoolsset{showonlyrefs=true}
\usepackage{bm}
\usepackage{bbm}
\usepackage{tensor}
\usepackage[colorlinks=true]{hyperref}
\hypersetup{colorlinks = true,
            citecolor = {blue},
            linkcolor = {red}}
\usepackage[all]{hypcap}
\usepackage[disable]{todonotes}%
\usepackage{color,soul}%
\newcommand{\m}[1]{\mathcal{#1}}%
\newcommand{\bb}[1]{\mathbb{#1}}%
\newcommand{\mrm}[1]{\mathrm{#1}}%
\renewcommand{\Re}{\mathrm{Re}}%
\renewcommand{\Im}{\mathrm{Im}}%
\newcommand{\RR}{\mathbb{R}}
\newcommand{\CC}{\mathbb{C}}%
\newcommand{\Alt}{\mathcal{A}}
\DeclareMathOperator{\Tr}{Tr}
\DeclareMathOperator{\End}{End}
\DeclareMathOperator{\Hom}{Hom}
\DeclareMathOperator{\rk}{rk}
\DeclareMathOperator{\range}{ran}
\newcommand{\del}{\partial}
\newcommand{\dd}{\mathrm{d}}%
\newcommand{\I}{\mkern1mu\mathrm{i}\mkern1mu}
\newcommand{\e}{\mathrm{e}}
\newcommand{\transpose}{\intercal}
\newcommand{\topdeg}{\mathrm{top}}
\newcommand{\id}{\mathbbm{1}}
\newcommand{\vol}{\mathrm{Vol}}
\newcommand{\sheaf}{\mathcal{O}}
\newcommand{\sing}{\mathrm{Sing}}
\newcommand{\tors}{\mathrm{Tors}}
\newcommand{\Todd}{\mathrm{Td}}
\newcommand{\chern}{\mathrm{c}}
\newcommand{\Chern}{\mathrm{ch}}
\newcommand{\curvform}{\mathcal{F}}
\newcommand{\Zdiff}{\mathcal{Z}}
\newcommand{\Zch}{Z}
\newcommand{\sym}{\mathrm{sym}}
\newcommand{\Hermmetric}{\mathcal{H}^+}
\def\Todd{\mathrm{Todd}}%
\def\End{\mathrm{End}}%
\def\rk{\mathrm{rk}}%
\def\PP{\mathbb{P}}%
\def\ch{\mathrm{ch}}%
\DeclarePairedDelimiter\abs{\lvert}{\rvert}
\newtheorem{thm}{Theorem}[section]%
\newtheorem{prop}[thm]{Proposition}%
\newtheorem{lemma}[thm]{Lemma}%
\newtheorem{cor}[thm]{Corollary}%
\newtheorem{conj}[thm]{Conjecture}%
\theoremstyle{definition}%
\newtheorem{definition}[thm]{Definition}%
\newtheorem{rmk}[thm]{Remark}
\newtheorem{exm}[thm]{Example}
\title{Z-critical equations for holomorphic vector bundles on K\"ahler surfaces}%
\date{}%
\author{Julien Keller and Carlo Scarpa}

\begin{document}

\maketitle

\begin{abstract}
    We prove that the existence of a $\Zch$-positive and $\Zch$-critical Hermitian metric on a rank~$2$ holomorphic vector bundle over a compact K\"ahler surface implies that the bundle is $\Zch$-stable. As particular cases, we obtain stability results for the deformed Hermitian Yang-Mills equation and the almost Hermite-Einstein equation for rank $2$ bundles over surfaces. We show examples of $\Zch$-(un)stable bundles and $\Zch$-critical metrics away from the large volume limit. 
\end{abstract}

\tableofcontents

\section{Introduction}


Let~$E\to X$ be a holomorphic vector bundle over the compact complex manifold~$X$, and assume that~$\omega$ is a K\"ahler form on~$X$. We consider a class of partial differential equations, called \emph{$\Zch$-critical equations}, to be solved for a Hermitian metric~$h$ on~$E$, that take the form
\begin{equation}\label{eq:Zcrit}
\Im(\e^{-\I\vartheta_E}\Zdiff(h))=0.
\end{equation}
The equation depends on~$\omega$ and the choice of a \emph{polynomial central charge} as defined by Bayer~\cite{Bayer_centralcharge}. Briefly, a polynomial central charge~$\Zch$ is defined by a vector of complex numbers~$\rho\in(\bb{C}^*)^n$ and a unitary class~$U\in H^\bullet(X,\bb{R})$, and associates to any subvariety~$V\subset X$ and any sheaf~$S\to V$ a complex number~$\Zch_V(S)$.

The~$\Zch$-critical equations have been introduced by Dervan, McCarthy, and Sektnan in~\cite{DervanMcCarthySektnan} as a possible differential-geometric counterpart to Bridgeland stability conditions. It is conjectured that the existence of solutions of~\eqref{eq:Zcrit} should be equivalent to an algebraic stability condition on the bundle~$E$, at least in certain regimes. This expectation has already been partially confirmed in various interesting cases. For a line bundle and a particular choice of central charge,~\eqref{eq:Zcrit} becomes the deformed Hermitian Yang-Mills equation (dHYM equation), for which it is known that the existence of solutions is equivalent to an algebraic positivity condition that is reminiscent of the Nakai-Moishezon criterion. This correspondence for the dHYM equation has attracted a lot of attention in recent years; for the sake of brevity, we refer the reader to~\cite{CollinsJacobYau_stability},~\cite{GaoChen_Jeq_dHYM},~\cite{Takahashi_dHYM_MoishezonCriterion} for an in-depth treatment of the subject. On higher-rank bundles, one of the main results of~\cite{DervanMcCarthySektnan} is that one has a correspondence between the existence of solutions of~\eqref{eq:Zcrit} and the \emph{$\Zch$-stability} of~$E$ in an asymptotic regime known as the \emph{large volume limit}. It is still an open question to link the existence of solutions of~\eqref{eq:Zcrit} to Bridgeland stability conditions. We refer to~\cite{McCarthy_thesis} for some general considerations in this direction and to the papers~\cite{CollinsYau_dHYMgeodesics},~\cite{CollinsShi_stabilitydHYM}, and~\cite{CollinsLoShiYau_stability_line_bundles} for an in-depth discussion of the possible issues with this correspondence for the dHYM equation on line bundles.

Establishing the existence of solutions of~\eqref{eq:Zcrit} poses in general exceptional difficulties, and most of the few results that are known are limited to low-rank and low-dimensional situations, where~$\Zch$-critical equations simplify substantially. A possible exception is given by the examples of dHYM connections over a Fano threefold appearing in the recent paper~\cite{Correa_dHYM_highrank}.
If~$X$ is a curve, the~$\Zch$-critical equation reduces to the Hermite-Einstein problem, for~$E$ of any rank. If instead~$X$ is a complex surface and~$\rk(E)=1$, the problem of the existence of solutions for any~$\Zch$-critical equation is essentially settled in~\cite[\S$2.3.3$]{DervanMcCarthySektnan}, at least under a mild positivity assumption called the \emph{volume form hypothesis}. For~$\rk(E)=1$, it seems likely that one can approach the~$\Zch$-critical equation following the study of the deformed Hermitian Yang-Mills equation. The possible presence of a non-vanishing unitary class (see Section~\ref{sec:Zcrit} for the definition) however greatly complicates the analysis even in the rank-$1$ case: it might not be possible to write the~$\Zch$-critical equation as a PDE for the eigenvalues of a Hermitian matrix, a feature of the dHYM equation that was crucial to develop the PDE theory of~\cite{GaoChen_Jeq_dHYM} to link the existence of solutions with an algebraic stability condition conjectured in~\cite{CollinsJacobYau_stability}.

In this paper we focus instead on the case when the base manifold~$X$ is a complex surface, and~$E$ is a vector bundle of rank greater than one. Our main goal is to refine the conjectural correspondence proposed in~\cite{DervanMcCarthySektnan} between the existence of~$\Zch$-critical metrics~$h\in\Hermmetric(E)$ and algebraic stability conditions on the pair~$(X,E)$. In particular, we show evidence in support of such a correspondence in \emph{non-asymptotic regimes}, at least for bundles of rank~$2$. A key role is played by a property of metrics ~$h\in\Hermmetric(E)$ called \emph{$\Zch$-positivity}. This is a notion of subsolution for the~$\Zch$-critical equation, see Section~\ref{sec:subsolutions} for the precise definition.

\begin{thm}\label{thm:Zcrit_stability}
    Let~$X$ be a compact K\"ahler surface, and let~$\Zch$ be a polynomial central charge. For any rank~$2$ indecomposable vector bundle~$E\to X$, if there exists~$h\in\Hermmetric(E)$ that is~$\Zch$-positive and solves the~$\Zch$-critical equation, then for any sub-bundle~$S\subset E$ such that~$0<\rk(S)<\rk(E)$ one has
    \begin{equation}
        \Im\left(\frac{\Zch_X(S)}{\Zch_X(E)}\right)<0.
    \end{equation}
    If the bundle moreover satisfies condition~\eqref{eq:Z_alphapositive} below, this inequality holds for any coherent saturated subsheaf~$S\subset E$ of rank~$0<\rk(S)<\rk(E)$.
\end{thm}
This is the first stability result for the existence of solutions to the~$\Zch$-critical equation, and the dHYM equation in particular, for bundles of rank greater than~$1$ in a non-asymptotic regime. Theorem~\ref{thm:Zcrit_stability} gives strong additional evidence for an algebraic characterisation of the existence of solutions to the~$\Zch$-critical equation, along the lines of~\cite[Conjecture~$1.6$]{DervanMcCarthySektnan}.

Our computations and some simple examples make it natural to propose a refinement of the conjectural stability picture of~\cite{DervanMcCarthySektnan} for the existence of~$\Zch$-critical metrics.
\begin{definition}\label{def:Zstable}
    Given a polynomial central charge~$\Zch$, a holomorphic vector bundle~$E$ over~$X$ is said to be \emph{$\Zch$-stable} (resp.~$\Zch$-\emph{semistable}) if, for any coherent torsion-free quotient~$Q$ of~$E$ of rank~$0<\rk(Q)<\rk(E)$,
    \begin{equation}\label{eq:Zstable}
        \Im\left(\frac{\Zch_X(Q)}{\Zch_X(E)}\right)>0\ (\mbox{resp. }\geq 0).
    \end{equation}
\end{definition}
One should not aspect, in general, that the existence of a~$\Zch$-critical metric implies~$\Zch$-stability in the absence of some positivity of the metric, see Remark~\ref{rmk:positivity_needed_stability} and Example~\ref{ex:not_positive_P2}. On the other hand, the existence of a~$\Zch$-positive metric on a bundle should allow us to link the existence of~$\Zch$-critical metrics and~$\Zch$-stability.
\begin{conj}\label{conj:Zstable}
    Let~$\Zch$ be a polynomial central charge on a K\"ahler manifold~$X$. Assume that~$E$ is an indecomposable holomorphic vector bundle on~$X$ that admits a~$\Zch$-positive Hermitian metric. Then there exists a~$\Zch$-positive solution of the~$\Zch$-critical equation if and only if~$E$ is~$\Zch$-stable.
\end{conj}
This should be compared with~\cite[Theorem~$1.2$]{CollinsJacobYau_stability} (see also~\cite{GaoChen_Jeq_dHYM}), where for the dHYM equation on a line bundle, it is shown that the existence of a subsolution is sufficient to show that there is a solution, under a \emph{supercritical phase condition}.

A natural question then is to find conditions on a bundle that might guarantee the existence of~$\Zch$-positive metrics. The case of the J-equation and the dHYM equations in rank~$1$ suggest that this condition should involve subvarieties of~$X$. Our second result shows that the existence of~$\Zch$-positive metrics implies conditions of this type.
\begin{definition}\label{def:ZpositiveZstable}
    Given a polynomial central charge~$\Zch$, a holomorphic vector bundle~$E$ over~$X$ is said to be \emph{$\Zch$-positive} if, for any analytic sub-variety~$V\subset X$ of dimension~$0<\dim V<\dim X$,
    \begin{equation}\label{eq:Zpositive}
        \Im\left(\frac{\Zch_V(E_{\restriction V})}{\Zch_X(E)}\right)>0.
    \end{equation}
\end{definition}
It will be important in what follows to also consider the inequality~\eqref{eq:Zpositive} for~$0$-di\-men\-sional analytic subvarieties~$V\subset X$, namely
\begin{equation}\label{eq:Z_alphapositive}
    \mbox{for any }x\in X,\quad \Im\left(\frac{\Zch_{\{x\}}(E_x)}{\Zch_X(E)}\right)>0.
\end{equation}
We refer to Remark~\ref{rmk:strong_Zpositive} for an alternative characterisation of condition~\eqref{eq:Z_alphapositive}.

\begin{thm}\label{thm:Zcrit_positivity}
    Let~$E\to X$ be a vector bundle on the compact K\"ahler surface~$X$, and fix a polynomial central charge~$\Zch$. If there exists a~$\Zch$-positive metric~$h\in\Hermmetric(E)$, then
    \begin{enumerate}[label=\arabic*.,ref=\thethm (\arabic*)]
        \item\label{thm:Zcrit_positivity_bundle}~$E$ is~$\Zch$-positive.
    \end{enumerate}
    If moreover~$E$ is a rank~$2$ bundle over a projective surface~$X$ that satisfies~\eqref{eq:Z_alphapositive}, then
    \begin{enumerate}[start=2,label=\arabic*.,ref=\thethm (\arabic*)]
        \item\label{thm:Zcrit_positivity_quotients}~$\Im\left(\frac{\Zch_V(Q)}{\Zch_X(E)}\right)>0$ for any~$1$-dimensional analytic (irreducible and reduced) subvariety~$V\subset X$ and any torsion-free quotient~$Q$ of~$E_{\restriction V}$ such that~$0<\rk(Q)<\rk(E)$.
    \end{enumerate}
\end{thm}
If the opposite inequality holds in~\eqref{eq:Z_alphapositive} we have analogous properties for subsheaves of~$E$ rather than quotients, see Remark~\ref{rmk:subbundles} below. Theorem~\ref{thm:Zcrit_positivity_bundle} is due to McCarthy, who proved it for codimension~$1$ analytic submanifolds of a K\"ahler manifold of arbitrary dimension, see~\cite[Theorem~$1.6$ and Theorem~$4.3.13$]{McCarthy_thesis}. For future reference, we will call the property in Theorem~\ref{thm:Zcrit_positivity_quotients} \emph{$\Zch$-positivity for quotients}.

We should also remark that it might be possible to obtain stronger results than Theorem~\ref{thm:Zcrit_positivity}, since the~$\Zch$-positivity of a metric on~$E$ is a priori stronger than the positivity condition we need to obtain each part of the statement, see Section~\ref{sec:positivity_stability}. Still, the recent characterisations for the existence of solutions of the J-equation~\cite{DatarPingali_numericalcriterion} and the dHYM equation in rank~$1$~\cite{Takahashi_dHYM_MoishezonCriterion} suggest that one might hope for a correspondence between the existence of~$\Zch$-positive metrics and positivity properties of the bundle, at least in the projective case.
\begin{conj}\label{conj:Zpositive}
    Let~$\Zch$ be a polynomial central charge on a projective surface~$X$. Assume that~$E$ is a~$\Zch$-positive rank~$2$ holomorphic vector bundle on~$X$ that satisfies~\eqref{eq:Z_alphapositive}. If for any~$1$-dimensional analytic subvariety~$V\subset X$ and any torsion-free quotient~$Q$ of~$E_{\restriction V}$
    \begin{equation}
        \Im\left(\frac{\Zch_V(Q)}{\Zch_X(E)}\right)>0,
    \end{equation}
    then~$E$ admits a~$\Zch$-positive metric~$h\in\Hermmetric(E)$.
\end{conj}
It is likely that these conjectural correspondences between the existence of (sub) solutions to the~$\Zch$-critical equation and~$\Zch$-stability will need some refinement, probably in the form of a ``supercritical phase condition''. Indeed, such a condition is crucial in the stability characterisation for the existence of solutions of the dHYM equation. For general central charges on line bundles over complex surfaces, we also know from~\cite[Theorem 2.27]{DervanMcCarthySektnan} that both conjectures hold under an additional positivity condition called the \emph{volume form hypothesis}. We will briefly discuss in Section~\ref{sec:examples} why a similar hypothesis might in fact be part of the appropriate condition to establish our conjectures, at least over projective surfaces.

The main observation leading to Theorem~\ref{thm:Zcrit_stability} is that the~$\Zch$-critical equation on a surface can be recast as a \emph{vector bundle Monge-Amp\`ere equation}, an equation that coincides with the usual complex Monge-Ampère equation in the case when the bundle has rank~$1$. This property was first noted for the deformed Hermitian Yang-Mills equation in rank~$1$ in~\cite{JacobYau_special_Lag}, and was then exploited in~\cite{DervanMcCarthySektnan} to show that a Nakai-Moishezon-type criterion characterises the existence of solutions to the~$\Zch$-critical equations in rank~$1$. In higher rank, a particular case of this phenomenon was noted by Takahashi in~\cite{Takahashi_Jeq_bundles} for the J-equation, which is the \emph{small radius limit} of the deformed Hermitian Yang-Mills equation. It is important to mention that~\cite{Takahashi_Jeq_bundles} established a result similar to Theorem~\ref{thm:Zcrit_stability} for the J-equation, again on rank~$2$ bundles.

The vector bundle Monge-Ampère equation was first introduced by Pingali in~\cite{Pingali_vbMA}. He showed that, under some positivity assumptions, the existence of solutions of the vector bundle Monge-Ampère equation implies a condition called \emph{Monge-Ampère stability}, see Section~\ref{sec:vbMA} for the definition. Theorem~\ref{thm:Zcrit_stability} will follow from a slight modification of Pingali's stability result:
\begin{thm}[~\cite{Pingali_vbMA}, Proposition~$3.1$]\label{thm:Pingali_stab}
Assume that~$E\to X$ is an indecomposable rank~$2$ holomorphic vector bundle over a compact K\"ahler surface, and that~$h\in\Hermmetric(E)$ solves
\begin{equation}\label{eq:vbMA}
    \curvform(h)^2=\eta\otimes\id_{E}
\end{equation}
for a volume form~$\eta$ on~$X$. If~$\Tr\curvform(h)>0$, then~$E$ is Monge-Amp\`ere stable in the sense of Definition~\ref{MAstabl}.
\end{thm}

The paper is organised as follows: in Section~\ref{sec:Zcrit} we collect the definitions and some background on the~$\Zch$-critical and the Monge-Ampère vector bundle equations. Section~\ref{sec:Zstab_vbMAstab} is the heart of the paper and contains the proof of Theorem~\ref{thm:Zcrit_stability} and Theorem~\ref{thm:Zcrit_positivity}. We also comment on other stability notions that are closely related to~$\Zch$-stability, and show an openness result for the existence of~$\Zch$-critical metrics that allows to obtain many new examples starting from Mumford stable or Monge-Ampère stable bundles. Section~\ref{sec:examples} contains some (non-)examples of~$\Zch$-positive and~$\Zch$-critical metrics on bundles of rank~$2$ and~$3$ on~$\PP^2$. We also propose a generalisation of Conjecture~\ref{conj:Zstable} that includes decomposable bundles, leading to a notion of~$\Zch$-polystability, and we prove it for decomposable rank~$2$ bundles over surfaces. Finally, we show that our results can be used to deduce stability results for the almost Hermite-Einstein equation in non-asymptotic cases, showing that there is a non-asymptotic analogue of Gieseker stability.

\paragraph*{Acknowledgements.}
The authors wish to thank Ruadhaí Dervan, Annamaria Ortu, Vamsi Pingali, Lars Martin Sektnan, and Sohaib Khalid for some helpful remarks and discussions related to the present work, and we thank Gonçalo Oliveira for a useful conversation regarding Remark~\ref{rmk:positivity_noncritical} and Tristan Collins for pointing out an inaccuracy. The second author would like to thank the Isaac Newton Institute for Mathematical Sciences, Cambridge, for support and hospitality during the programme ``New equivariant methods in algebraic and differential geometry'' where work on this paper was undertaken. This work was supported by the EPSRC grant no EP/R014604/1, an NSERC Discovery Grant, and a FRQNT grant (team research project program).

\section{The Z-critical and vector bundle Monge-Ampère equations}\label{sec:Zcrit}\label{sec:vbMA}

In this section we recall the definitions and some basic properties of the~$\Zch$-critical and vector bundle Monge-Ampère equations. We refer the reader to~\cite{DervanMcCarthySektnan} and~\cite{McCarthy_thesis} for a more in-depth discussion of the former equation, and to~\cite{Pingali_vbMA} for the latter.

Let~$E\to X$ be a holomorphic vector bundle over a compact complex manifold~$X$ of complex dimension~$n$, and assume that~$X$ carries a K\"ahler metric~$\omega$. Given any affine connection~$D$ on~$E$, we denote the~$(1,0)$ and~$(0,1)$ parts of the connection by~$D'$ and~$D''$ respectively, and the curvature by~$F(D)\in\Alt^{1,1}(\End E)$. It will also be convenient to also define the \emph{normalised curvature form} of~$D$ as
\begin{equation}
\curvform(D):=\frac{\I}{2\pi}F(D).
\end{equation}
Given a Hermitian metric~$h$ on~$E$, we will also denote by~$\curvform(h)$ the normalised curvature of the Chern connection defined by~$h$ and the holomorphic structure of~$E$, which is a~$(1,1)$-form on~$X$ with values in the self-adjoint (with respect to~$h$) endormorphisms of~$E$. This normalisation is chosen so that the differential form
\begin{equation}
\Tr\e^{\curvform(h)}:=\Tr\left(\sum_{j=0}^{\rk(E)}\frac{1}{j!}\curvform(h)^j\right)
\end{equation}
represents the total Chern character of~$E$,~$\Chern(E)$. We denote by~$\Chern_d(E)$ the degree-$2d$ component of~$\Chern(E)$, which is represented by the~$(2d,2d)$-form~$\Tr\curvform^d/d!$.

The~$\Zch$-critical equation~\eqref{eq:Zcrit} as defined in~\cite{DervanMcCarthySektnan} depends on the choice of:
\begin{itemize}
\item a \emph{stability vector}~$\rho\in\CC^{n+1}$, i.e.~$n+1$ nonzero complex numbers~$(\rho_0,\dots,\rho_n)$ such that~$\Im(\rho_j/\rho_{j+1})>0$ for~$0\leq j<n$;
\item the representative~$u\in\Alt^{\bullet}(X,\RR)$ of a \emph{unipotent class}~$U=1+\sum_jU_j\in H^{\bullet}(X,\RR)$ such that~$U_j\in H^{j,j}(X,\RR)$ for each~$j$. 
\end{itemize}
Given these objects, Dervan-McCarthy-Sektnan~\cite{DervanMcCarthySektnan} define a~$\Alt^{n,n}(\End E)$-valued differential operator,
\begin{equation}\label{eq:Zdiff_operator}
\Zdiff(h)=\left[\left(\sum_{j=0}^n\rho_j\omega^j\right)\wedge u\wedge\e^{\curvform(h)}\right]^{\topdeg},
\end{equation}
where~$[\dots]^{\topdeg}$ indicates that one only has to consider the maximal-degree part of a differential form. This data also defines a \emph{polynomial central charge} as in~\cite{Bayer_centralcharge}
\begin{equation}\label{eq:centralcharge_definition}
\Zch_X(E)=\int_X\Tr\Zdiff(h),
\end{equation}
that does not depend on the choices of~$u\in U$ and~$h\in \Hermmetric(E)$. Assuming that~$\Zch(E)$ is nonzero, the \emph{$\Zch$-critical equation} for a Hermitian metric~$h\in\Hermmetric(E)$ (or the associated Chern connection~$D$) is
\begin{equation}
\Im(\e^{-\I\vartheta_E}\Zdiff(D))=0
\end{equation}
where the phase angle~$\vartheta_E$ is determined modulo~$2\pi$ through integration, i.e.
\begin{equation}
\Zch(E)=\e^{\I\vartheta_E}\RR_{>0}.
\end{equation}
\begin{rmk}\label{rmk:largevolumelimit}
    There are some minor differences between how the~$\Zch$-critical equation was introduced in~\cite{DervanMcCarthySektnan} and the one presented here, mainly due to the fact that in the original paper the authors focus on an asymptotic regime known as the \emph{large volume limit}. In that context, one rescales the K\"ahler form by~$\omega\mapsto k\,\omega$ and is interested in the properties of the~$\Zch$-critical equation~\eqref{eq:Zcrit} for~$k\gg0$. The stability vectors in~\cite{DervanMcCarthySektnan} then are required to satisfy different conditions than the ones we consider: they impose~$\Im(\rho_n)>0$ and~$\Im(\rho_{n-1}/\rho_{n})>0$, rather than~$\Im(\rho_j/\rho_{j+1})>0$ for~$0\leq j<n$. The condition~$\Im(\rho_n)>0$ however is just a choice of normalisation, one can ensure this by rotating the whole stability vector without affecting the stability of the bundle nor the existence of critical metrics. The reason why in~\cite{DervanMcCarthySektnan} only the last condition of~$\Im(\rho_{j}/\rho_{j+1})>0$ is required, is that this is the minimum necessary to guarantee that any asymptotically~$\Zch$-stable bundle (see Section~\ref{sec:stabilities}) is Mumford semistable. From our point of view however, it is more natural to have the same assumptions as in the original definition of a polynomial stability condition, see~\cite{Bayer_centralcharge}.
\end{rmk}
\begin{rmk}
Here, as in~\cite{DervanMcCarthySektnan}, by ``imaginary part''~$\Im(A)$ of an endomorphism-valued form~$A\in\Alt^{\bullet}(X,\End E)$ we mean ($-\I$ times) the \emph{anti-self-adjoint} component of the endomorphism,~$\Im(A)=\frac{1}{2\I}\left(A-A^*\right)$ with respect to the Hermitian metric on~$E$. Similarly, the ``real part'' indicates the self-adjoint component. When the Hermitian bundle has rank one, these indeed coincide with the real and imaginary parts of a complex-valued~$1$-form. With this convention then, if~$D$ is the Chern connection of a Hermitian bundle~$(E,h)$, one has~$\Im(\curvform(D))=0$ and~$\Re(\curvform(D))=\curvform(D)$.
\end{rmk}

When~$X$ is a complex surface, the general~$\Zch$-critical operator is, for~$h\in\Hermmetric(E)$,
\begin{equation}\label{eq:Zdiff_surf_general}
\begin{split}
    \Zdiff(h)=&\left[(\rho_0+\rho_1\omega+\rho_2\omega^2)\wedge(1+u_1+u_2)\wedge\left(\id_E+\curvform(h)+\frac{1}{2}\curvform(h)^2\right)\right]^\topdeg\\
    =&\rho_0u_2+\rho_1\omega\wedge u_1+\rho_2\omega^2+\left(\rho_0u_1+\rho_1\omega\right)\wedge\curvform(h)+\frac{1}{2}\rho_0\curvform(h)^2
\end{split}
\end{equation}
and the corresponding~$\Zch$-charge is given by integrating the trace of~\eqref{eq:Zdiff_surf_general} on~$X$:
\begin{equation}\label{eq:Zcritical_surface_original}
    \Zch_X(E)=\left(\rho_0U_2+\rho_1U_1.[\omega]+\rho_2[\omega]^2\right)\rk(E)+\left(\rho_0U_1+\rho_1[\omega]\right).\Chern_1(E)+\rho_0\Chern_2(E).
\end{equation}
The expression~\eqref{eq:Zcritical_surface_original} for the central charge allows us to define~$\Zch_X(E)$ for a coherent sheaf that is not necessarily locally free. Similarly, if~$E\to V$ is a coherent sheaf on a (reduced, irreducible) curve~$V\subset X$, we define~$\Zch_V(E)$ as
\begin{equation}
    \Zch_V(E)=\rho_1\,[\omega].V\,\rk(E)+\rho_0\,U_1.V\,\rk(E)+\rho_0\,\chern_1(E).V
\end{equation}
where~$\chern_1(E).V$ denotes the degree of~$E$ over~$V$, defined as
\begin{equation}
    \deg_V(E):=\chi(V,E)-\chi(V,\mathcal{O}_V)=\chi(V,E)+g-1,
\end{equation}
where~$g$ is the genus of~$V$. Note that in the smooth projective setting, one can also use a finite resolution by locally free sheaves to define the Chern classes (and thus the degree over a smooth curve) of a coherent sheaf, see~\cite[p. 29]{Friedman}.

\smallbreak

As~$\curvform(h)$ is self-adjoint with respect to the metric~$h$ and~$U$ is a real class, the~$\Zch$-critical equation can be written as
\begin{equation}\label{eq:Zcritical_surface}
    \alpha\,\curvform(h)^2+\beta\wedge\curvform(h)+\gamma\otimes\id_E=0
\end{equation}
where~$\alpha\in\RR$,~$\beta\in\Alt^{1,1}(X,\RR)$ and~$\gamma\in\Alt^{2,2}(X,\RR)$ are defined as
\begin{equation}\label{eq:abc_coeff}
    \begin{dcases}
        \alpha=\frac{1}{2}\Im\left(\e^{-\I\vartheta_E}\rho_0\right)\\
        \beta=\Im\left(\e^{-\I\vartheta_E}\left(\rho_0u_1+\rho_1\omega\right)\right)\\
        \gamma=\Im\left(\e^{-\I\vartheta_E}\left(\rho_0u_2+\rho_1\omega\wedge u_1+\rho_2\omega^2\right)\right).
    \end{dcases}
\end{equation}
Note that the coefficients~$\beta$ and~$\gamma$ are~$\wedge$-commuting with~$\curvform(h)$, so the~$\Zch$-critical equation~\eqref{eq:Zcritical_surface} is equivalent to
\begin{equation}\label{eq:Zcrit_vbMA}
    \left(\curvform(h)+\frac{\beta}{2\alpha}\otimes\id_E\right)^2=\left(\left(\frac{\beta}{2\alpha}\right)^2-\frac{\gamma}{\alpha}\right)\otimes\id_E.
\end{equation}
We are assuming here that~$\alpha\not=0$, the case~$\alpha=0$ essentially reduces to the Hermite-Einstein problem and will be treated separately in Section~\ref{sec:alphazero}.
\begin{rmk}\label{rmk:strong_Zpositive}
    The condition~$\alpha>0$ can also be rephrased as~$\Im\left(\overline{\Zch_X(E)}\rho_0\right)>0$.
    Note that~$\Zch_{\{x\}}(E_x)=\rk(E)\rho_0$ for any point~$x\in X$, hence \emph{a bundle satisfies condition~\eqref{eq:Z_alphapositive} if and only if~$\alpha>0$}. This observation will be used repeatedly in what follows, in particular for the proof of Theorem~\ref{thm:Zcrit_positivity}.
\end{rmk}

The simple expression~\eqref{eq:Zcrit_vbMA} for the~$\Zch$-critical equation shows
\begin{lemma}\label{lemma:equation_equivalence}
    With the previous notation, let~$\eta:=(\beta/2\alpha)^2-\gamma/\alpha$. Then, the~$\Zch$-critical equation is equivalent to the \emph{twisted vector bundle Monge-Ampère equation}
    \begin{equation}\label{eq:vbMA_twisted}
        \left(\curvform(h)+\frac{\beta}{2\alpha}\otimes\id_E\right)^2=\eta\otimes\id_E.
    \end{equation}
\end{lemma}
When~$\eta$ is a volume form and~$\beta=0$ (or~$\beta/2\alpha\in\chern_1(N)$ for a line bundle~$N$), this equation was introduced and studied by Pingali in~\cite{Pingali_vbMA}. As we mentioned in the Introduction, the existence of solutions of~\eqref{eq:vbMA} is tied to the \emph{Monge-Ampère stability} of the bundle, at least in certain regimes, see Theorem~\ref{thm:Pingali_stab}. We introduce a version of this condition that takes into account the possible presence of a non-vanishing form~$\beta$. It will be useful in what follows to let
\begin{equation}
    \curvform_{\alpha,\beta}(h):=\curvform(h)+\frac{\beta}{2\alpha}\otimes\id_E.
\end{equation}
Generalising Pingali's Monge-Ampère slope, we define for~$\vartheta:=[\beta/2\alpha]\in H^{1,1}(X,\RR)$
\begin{equation}
    \mu_{MA,\vartheta}(E):=\frac{\Chern_2(E)}{\rk{E}}+\frac{\Chern_1(E).\vartheta}{\rk(E)}=\mu_{MA}(E)+\mu_{\vartheta}(E)
\end{equation}
where the Chern characters of~$S$ are defined through a vector bundle resolution of~$S$, and~$\mu_{MA}$ denotes the original Monge-Ampère slope of~\cite{Pingali_vbMA}.
\begin{definition}[\cite{Pingali_vbMA}]\label{MAstabl}
    For a real~$(1,1)$-class~$\vartheta$ and a vector bundle~$E$ on a K\"ahler surface~$X$, we say that~$E$ is \emph{$\vartheta$-Monge-Ampère stable} with respect to~$\vartheta$ if for every coherent saturated subsheaf~$S\subset E$ such that~$0<\rk(S)<\rk(E)$ we have
    \begin{equation}\label{eq:MA_stability}
        \mu_{MA,\vartheta}(S)<\mu_{MA,\vartheta}(E).
    \end{equation}
\end{definition}
In the next section, we will show that this twisted Monge-Amp\`ere stability is in fact equivalent to~$\Zch$-stability under the correspondence of Lemma~\ref{lemma:equation_equivalence}, at least when~$\alpha>0$. Note also that the vector bundle Monge-Ampère equation for a fixed vector bundle~$E$ can in fact be recast as a~$\Zch$-critical equation by choosing a stability vector~$\rho$ that depends on the Chern numbers of~$E$. More explicitly, let~$d=\deg(E)$ and~$c_2=\Chern_2(E)=\rk(E)$, and consider the stability vector~$\rho=\big(-1,\,c_2\I,\,1+c_2(1-d)\I\big)$. Then if~$\Zch^\rho$ is the polynomial central charge defined by~$\rho$ and a trivial unitary class, the~$\Zch^\rho$-critical equation is equivalent to the vector bundle Monge-Ampère equation~$\curvform(h)^2=2\,\omega^2\otimes\id_E$, for which the condition~$c_2=\rk(E)$ is necessary for the existence of solutions.

This shows that the (twisted) vector bundle Monge-Ampère equation is, in some sense, the main equation of interest for vector bundles on K\"ahler surfaces. The~$\Zch$-critical equation and the vector bundle Monge-Ampère equation seem otherwise unrelated on higher dimensional manifolds. We hope this work will stimulate further research on the vector bundle Monge-Ampère equation.

\subsection{Some positivity conditions for Hermitian metrics}\label{sec:subsolutions}

A problematic feature of both the vector bundle Monge-Ampère equation~\eqref{eq:Zcrit} and the~$\Zch$-critical equation~\eqref{eq:vbMA} is that they might fail to be elliptic. To characterise the situations in which the~$\Zch$-critical equation is elliptic at least near a solution, we can consider either the linearisation of the operator~$D\mapsto\e^{\curvform(D)}$, defined on the space of affine connections compatible with a fixed Hermitian metric~$h_0$, or the linearisation of~$h\mapsto\e^{\curvform(h)}$ defined on~$\Hermmetric(E)$. When we take a path of Hermitian metrics~$h_t$ or a path of connections~$D_t$ and consider the variation of~$\e^{\curvform_t}$, one finds over a complex dimension~$n$ manifold, 
\begin{equation}
\partial_{t=0}\e^{\curvform_t}=\partial_{t=0}\sum_{j=1}^n\frac{1}{j!}\curvform_t^j=\sum_{j=1}^n\frac{1}{j!}\sum_{p=1}^{j}\curvform_0^{p-1}\left(\partial_{t=0}\curvform_t\right)\curvform_0^{j-p}.
\end{equation}

We introduce a multilinear product on~$\Alt^\bullet(\End E)$ as
\begin{equation}
\left[A_1\wedge\dots\wedge A_j\right]_{\sym}:=\frac{1}{j!}\sum_{\sigma\in S_j}(-1)^{\mathrm{gr\,sgn}(\sigma)}A_{\sigma(1)}\wedge\dots\wedge A_{\sigma(j)}
\end{equation}
where the graded sign of a permutation,~$\mathrm{gr\,sgn}(\sigma)$, is defined as the sign obtained by permuting~$A_1\wedge\dots A_j$ to~$A_{\sigma(1)}\wedge\dots\wedge A_{\sigma(j)}$ as differential forms (so, ignoring~$\End E$ factors). Then the derivative of~$\e^{\curvform_t}$ can be rewritten as
\begin{equation}
\partial_{t=0}\e^{\curvform_t}=\sum_{j=1}^n\frac{1}{(j-1)!}\big[\underbrace{\curvform_0\wedge\dots\wedge\curvform_0}_{j-1\mbox{ times}}\wedge\partial_{t=0}\curvform_t\big]_{\sym}=\big[\e^{\curvform_t}\wedge\partial_{t=0}\curvform_t\big]_{\sym}
\end{equation}
so that the derivative of~$\Zdiff$, say along a path of connections~$D_t$, can be written as
\begin{equation}
\partial_{t=0}\Zdiff(D_t)=\left[\left[\left(\sum\nolimits_j\rho_j\omega^j\right)\wedge u\wedge\e^{\curvform(h_t)}\right]^{n-1,n-1}\wedge\partial_{t=0}\curvform(D_t)\right]_{\sym}.
\end{equation}
It seems reasonable then to let~$\Zdiff'(D):=\left[\left(\sum\nolimits_j\rho_j\omega^j\right)\wedge u\wedge\e^{\curvform(D)}\right]^{n-1,n-1}$. The differential of~$\Zdiff$ along a path of holomorphic structures~$D''_t$ on~$E$ generated by the action of an element of the complexified gauge group~$V\in\I\Alt^0(X,\End(E,h))$ then becomes
\begin{equation}\label{eq:derivative_Zdiff}
\partial_{t=0}\Zdiff(D(h,D''_t))=\frac{\I}{2\pi}\left[\Zdiff(D)'\wedge\,(D''D'-D'D'')V\right]_{\sym}.
\end{equation}
This discussion suggests to impose an additional positivity condition on Hermitian metrics on~$E$ to guarantee the ellipticity of the equation. 
\begin{definition}\label{ZposMetric}
Given a polynomial central charge~$Z$ on a holomorphic vector bundle~$E$, we say that a Hermitian metric~$h$ (or its Chern connection) is \emph{$\Zch$-positive} (also called (strongly) \emph{$\Zch$-subsolution} in~\cite{DervanMcCarthySektnan}) if the~$2(n-1)$~$\End(E)$-valued form~$\Im(\e^{-\I\vartheta_E}\Zdiff'(h))$ is positive definite; i.e. for any~$x\in X$ and any~$\xi\in T^{0,1}_x{}^*X\times\End(E_x)$
\begin{equation}\label{eq:subsolution}
\I\Tr\left[\Im\left(\e^{-\I\vartheta_E}\Zdiff'(h)\right)\wedge\xi^*\wedge\xi\right]_{\sym}>0
\end{equation}
where~$\Zdiff'(h)$ is the formal derivative of~$\Zdiff(h)$ with respect to~$\curvform(h)$.
\end{definition}
This subsolution condition is crucial for the moment map interpretation of the~$\Zch$-critical equation in~\cite{DervanMcCarthySektnan}; indeed, the equation can be shown to be coming from a Hamiltonian action on the space of connections that are unitary with respect to a fixed Hermitian metric~$h_0$ and \emph{satisfy the subsolution condition}; the symplectic form on this space is given by the Hermitian pairing
\begin{equation}\label{eq:momentmap_pairing}
\langle a,b\rangle=-\I\int_X\Tr\left[\Im\left(\e^{-\I\vartheta_E}\Zdiff'(h)\right)\wedge a\wedge b^*\right]_{\sym}>0
\end{equation}
for~$a,b\in\Alt^{0,1}\End E$, which is positive provided that the Chern connection of~$h$ is a~$\Zch$-subsolution: this also partially motivates our alternative naming of \emph{$\Zch$-positive} metric for~$h\in\Hermmetric(E)$ satisfying~\eqref{eq:subsolution}.

There is a similar positivity notion for the vector bundle Monge-Ampère equation, that we adapt from~\cite{Pingali_vbMA} to include the possible presence of a nontrivial twist.
\begin{definition}
    Given a nonzero~$\alpha\in\RR$ and a~$(1,1)$-form~$\beta$, a metric~$h\in\Hermmetric(E)$ is said to be \emph{twisted Monge-Ampère positive} with respect to~$\alpha$ and~$\beta$ if for any~$x\in X$ and any~$\xi\in T^{0,1}_x{}^*X\times\End(E_x)$
    \begin{equation}\label{eq:MA_positivity}
        \sum_{k=0}^{n-1}\I\Tr\left[ \xi^*\wedge\curvform_{\alpha,\beta}(h)^k\wedge \xi\wedge\curvform_{\alpha,\beta}(h)^{n-k-1} \right]>0.
    \end{equation}
\end{definition}
Note that these positivity conditions are trivial if~$X$ is a curve. The direct computation shows that in the~$2$-dimensional case,~$\Zch$-positivity and twisted Monge-Ampère positivity are equivalent under the correspondence of Lemma~\ref{lemma:equation_equivalence}, if~$\alpha$ is positive. Indeed, on a complex surface we have, with the notation of~\eqref{eq:abc_coeff},
\begin{equation}
    \Im\left(\e^{-\I\vartheta_E}\Zdiff'\right)=\Im\left(\e^{-\I\vartheta_E}\left(\rho_0\curvform+(\rho_0u_1+\rho_1\omega)\otimes\id_E\right)\right)=2\alpha\curvform+\beta\otimes\id_E
\end{equation}
so that
\begin{equation}\label{eq:positivity_comparison}
\begin{split}
    \I\Tr\left[\Im\left(\e^{-\I\vartheta_E}\Zdiff'(h)\right)\wedge\xi^*\wedge\xi\right]_{\sym}=&\I\Tr\left[\left(2\alpha\,\curvform(h)+\beta\otimes\id_E\right)\wedge\xi^*\wedge\xi\right]_{\sym}\\
    =&\alpha\I\Tr\left[\xi^*\wedge\xi\wedge\curvform_{\alpha,\beta}(h)+\xi^*\wedge\curvform_{\alpha,\beta}(h)\wedge\xi\right].
\end{split}
\end{equation}
If instead~$\alpha<0$, then~$h$ is~$\Zch$-positive if and only if it is Monge-Ampère \emph{negative}, i.e. for any~$x\in X$ and any nonzero~$\xi\in T^{0,1}_x{}^*X\times\End(E_p)$
\begin{equation}
    \I\Tr\left[\xi^*\wedge\xi\wedge\curvform_{\alpha,\beta}(h)+\xi^*\wedge\curvform_{\alpha,\beta}(h)\wedge\xi\right]<0.
\end{equation}

A consequence of this discussion is that, under the positivity condition~\eqref{eq:Z_alphapositive}, any result obtained for the vector bundle Monge-Ampère equation for a bundle on a K\"ahler surface can be almost immediately translated to a statement for the~$\Zch$-critical equation. As an example of this phenomenon, we can obtain a Chern class inequality for bundles that admit a~$\Zch$-critical metric from the analogous inequality in~\cite[Theorem~$1.2$]{Pingali_vbMA}.
\begin{prop}\label{prop:Bogo-analytic}
    Let~$X$ be a compact K\"ahler surface. For any polynomial central charge~$\Zch$ and any rank~$2$ bundle~$E\to X$ that satisfies~\eqref{eq:Z_alphapositive}, if~$E$ admits a~$\Zch$-positive and~$\Zch$-critical Hermitian metric and satisfies the volume form hypothesis~\eqref{eq:volumeform_hyp} then one has the Bogomolov inequality
    \begin{equation}\label{eq:Bogo}
        \chern_1(E)^2\leq 4\chern_2(E).
    \end{equation}
    Moreover, the equality holds if and only if~$E$ is projectively flat.
\end{prop}
\begin{proof}
    Under our hypothesis, the~$\Zch$-positivity and the (twisted) vector bundle Monge-Ampère positivity of~$h$ are equivalent, so it is sufficient to show the inequality if~$h$ is a solution of the twisted Monge-Ampère equation and is twisted Monge-Ampère positive. This is essentially already shown in the proof of~\cite[Proposition~$3.12$]{Pingali_vbMA}: indeed, the argument in~\cite[Proposition~$3.12$]{Pingali_vbMA} goes through almost verbatim when substituting the twisted curvature endomorphism~$\curvform(h)+\beta/2\alpha\otimes\id_E$ for~$\Theta$. 
\end{proof}

\section{Z-stability and Monge-Ampère stability}\label{sec:Zstab_vbMAstab}

Our first result allows us to compare the inequalities~\eqref{eq:MA_stability} and~\eqref{eq:Zstable}. As a corollary, we see that~$\Zch$-stability and twisted Monge-Amp\`ere stability are equivalent if~$\alpha>0$. All the results in this section hold under the assumption that~$X$ is a surface.
\begin{prop}\label{prop:stability_comparison}
    For any subsheaf~$S\subset E$, under the correspondence of Lemma~\ref{lemma:equation_equivalence} we have
    \begin{equation}
        \Im\left(\Zch_X(S)\,\overline{\Zch_X(E)}\right)=2\alpha\,\rk(S)\abs{\Zch_X(E)}\big(\mu_{MA,\vartheta}(S)-\mu_{MA,\vartheta}(E)\big).
    \end{equation}
\end{prop}
\begin{proof}
    It is just a matter of carefully computing each term. To start, recall from the definition that
    \begin{equation}\label{eq:MA_slopes'}
        \mu_{MA,\vartheta}(S)-\mu_{MA,\vartheta}(E)=\frac{\Chern_2(S)}{\rk(S)}-\frac{\Chern_2(E)}{\rk(E)}+\left(\frac{\Chern_1(S)}{\rk(S)}-\frac{\Chern_1(E)}{\rk(E)}\right).\frac{[\beta]}{2\alpha}.
    \end{equation}
    Note from~\eqref{eq:abc_coeff} that the class of~$\beta$ is
    \begin{equation}
        [\beta]=\Im\left(\e^{-\I\vartheta_E}\left(\rho_0U_1+\rho_1[\omega]\right)\right)=\abs{\Zch_X(E)}^{-1}\Im\left(\overline{\Zch_X(E)}(\rho_0U_1+\rho_1[\omega])\right).
    \end{equation}
    We proceed to the computation of~$\Im\left(\Zch_X(S)\,\overline{\Zch_X(E)}\right)$. By definition, we have
    \begin{equation}
        \begin{split}
            \Zch_X(E)=\left(\rho_0U_2+\rho_1U_1.[\omega]+\rho_2[\omega]^2\right)\rk(E)+\left(\rho_0U_1+\rho_1[\omega]\right).\Chern_1(E)+\rho_0\Chern_2(E),\\
         \Zch_X(S)=\left(\rho_0U_2+\rho_1U_1.[\omega]+\rho_2[\omega]^2\right)\rk(S)+\left(\rho_0U_1+\rho_1[\omega]\right).\Chern_1(S)+\rho_0\Chern_2(S).
        \end{split}
    \end{equation}
    The direct computation gives
    \begin{equation}\label{eq:imaginaryratio_subsheaf}
        \begin{split}
            \Im\left(\Zch_X(S)\,\overline{\Zch_X(E)}\right)=&\Im\left[\overline{\Zch_X(E)}\left(\rho_0U_2+\rho_1U_1.[\omega]+\rho_2[\omega]^2\right)\right]\rk(S)+\\
            &+\Im\left[\overline{\Zch_X(E)}\left(\rho_0U_1+\rho_1[\omega]\right)\right].\Chern_1(S)+\Im\left[\overline{\Zch_X(E)}\rho_0\right]\Chern_2(S).
        \end{split}
    \end{equation}
    As the only non-real quantities in~$\Zch_X(E)$ are the coefficients~$\rho_i$, one finds
    \begin{equation}
        \begin{split}
            &\Im\left[\overline{\Zch_X(E)}\left(\rho_0U_2+\rho_1U_1.[\omega]+\rho_2[\omega]^2\right)\right]=\\
            =&\frac{1}{\rk(E)}\Im\left[\overline{\Zch_X(E)}\Zch_X(E)\right]-\frac{1}{\rk(E)}\Im\left[\overline{\Zch_X(E)}\Big(\left(\rho_0U_1+\rho_1[\omega]\right).\Chern_1(E)+\rho_0\Chern_2(E)\Big)\right]=\\
            =&-\frac{1}{\rk(E)}\Im\left[\overline{\Zch_X(E)}\left(\rho_0U_1+\rho_1[\omega]\right).\Chern_1(E)+\overline{
            \Zch_X(E)}\rho_0\Chern_2(E)\right]=\\
            =&-\frac{1}{\rk(E)}\Im\left[\overline{\Zch_X(E)}\left(\rho_0U_1+\rho_1[\omega]\right)\right].\Chern_1(E)-\frac{1}{\rk(E)}\Im\left[\overline{\Zch_X(E)}\rho_0\right]\Chern_2(E).
        \end{split}
    \end{equation}
    We substitute this into~\eqref{eq:imaginaryratio_subsheaf} to obtain
    \begin{equation}\label{eq:stability_comparison}
        \begin{split}
            \Im\left(\Zch_X(S)\,\overline{\Zch_X(E)}\right)=&\Im\left[\overline{\Zch_X(E)}\left(\rho_0U_1+\rho_1[\omega]\right)\right].\left(-\frac{\rk(S)}{\rk(E)}\Chern_1(E)+\Chern_1(S)\right)\\
            &+\Im\left[\overline{\Zch_X(E)}\rho_0\right]\left(-\frac{\rk(S)}{\rk(E)}\Chern_2(E)+\Chern_2(S)\right)
        \end{split}
    \end{equation}
    and the right-hand side can be rewritten as
    \begin{equation}
        \rk(S)\abs{\Zch_X(E)}\left(\frac{\Chern_1(S)}{\rk(S)}-\frac{\Chern_1(E)}{\rk(E)}\right).[\beta]+\rk(S)\abs{\Zch_X(E)}2\alpha\left(\frac{\Chern_2(S)}{\rk(S)}-\frac{\Chern_2(E)}{\rk(E)}\right).\qedhere
    \end{equation}
\end{proof}

Pingali's Theorem~\ref{thm:Pingali_stab} holds assuming that~$\eta>0$,~$\beta=0$ and~$\Tr\curvform(h)>0$. In the case when~$\beta/2\alpha\in\chern_1(N)$ for a line bundle~$N$, we can reduce to this situation by considering the bundle~$E'=E\otimes N$, since there is a solution of the twisted vector bundle Monge-Amp\`ere equation on~$E$ if and only if there is a solution of the usual vector bundle Monge-Ampère equation on~$E'$. In our setting however the ``rationality assumption''~$\beta\in\chern_1(N)$ is not very natural, as it is not satisfied by many important charges, and we will show that indeed this hypothesis is not necessary.

Still, in view of Theorem~\ref{thm:Pingali_stab}, it is natural to consider the condition
\begin{equation}\label{eq:curv_positivitycondition}
    2\alpha\Tr\curvform(h)+\rk(E)\beta>0.
\end{equation}
We will prove shortly that any~$\Zch$-positive metric, not necessarily solving the~$\Zch$-critical equation, satisfies~\eqref{eq:curv_positivitycondition}. An interesting phenomenon is that the existence of \emph{some}~$h\in\Hermmetric(E)$ satisfying~\eqref{eq:curv_positivitycondition} is guaranteed by the inequality~\eqref{eq:Zpositive}.
\begin{lemma}[\cite{McCarthy_thesis}]\label{lemma:subsol_positiveclass}
    Let~$\Zch$ be a polynomial central charge, and let~$E\to X$ be a holomorphic vector bundle on the K\"ahler surface~$X$. If~$h\in\Hermmetric(E)$ is~$\Zch$-positive, then~$E$ is~$\Zch$-positive.
\end{lemma}
\begin{proof}
    We use the notation of~\eqref{eq:abc_coeff}. Recall that a metric~$h\in\Hermmetric(E)$ is~$\Zch$-positive if for any~$x\in X$ and any~$\xi\in T^{0,1}_x{}^*X\times\End(E_p)$
    \begin{equation}
        \I\Tr\left[ \xi^*\wedge\xi\wedge\left(2\alpha\,\curvform+\beta\otimes\id_E\right)+\xi^*\wedge\left(2\alpha\,\curvform+\beta\otimes\id_E\right)\wedge \xi\right]>0.
    \end{equation}
    Hence, if we choose~$\xi=v\otimes\id_{E}$ for some~$v\in T^{0,1}_x{}^*X$ we get
    \begin{equation}
        \I\bar{v}\wedge v\wedge\Tr\left[2\alpha\,\curvform+\beta\otimes\id_E\right]>0
    \end{equation}
    and as this holds for any choice of~$v$,~$2\alpha\Tr\curvform(h)+\rk(E)\beta$ is a positive~$(1,1)$-form.

    We now claim that: \emph{the cohomology class~$2\alpha\,\Chern_1(E)+\rk(E)[\beta]$ is positive if and only if~$E$ is~$\Zch$-positive}.
    
    Recall that the Demailly-P\u{a}un extension of the Nakai-Moishezon criterion~\cite{Demailly_Paun} (see also~\cite{Buchdal_surfaces} for the case of surfaces) implies that a class~$[\chi]$ is positive if and only if~$[\chi].[V]>0$ for every curve~$V\subset X$. Hence,~$2\alpha\,\Chern_1(E)+\rk(E)[\beta]>0$ if and only if for every~$V\subset X$
    \begin{equation}
        \left(2\alpha\,\Chern_1(E)+\rk(E)[\beta]\right).[V]>0
    \end{equation}
    which translates to, by definition of~$\beta$,
    \begin{equation}
        \left(2\alpha\,\Chern_1(E)+\frac{\rk(E)}{\abs{\Zch_X(E)}}\Im\left(\overline{\Zch_X(E)}\left(\rho_0U_1+\rho_1[\omega]\right)\right)\right).[V]>0.
    \end{equation}
    On the other hand, having fixed a curve~$V$,~$\Im\left(\Zch_V(E)\,\Zch_X(E)^{-1}\right)>0$ is equivalent to
    \begin{equation}
        \Im\left(\overline{\Zch_X(E)}\Zch_V(E)\right)>0
    \end{equation}
    and~$\Zch_X(E)$ is given by~\eqref{eq:Zcritical_surface_original}, while integrating the charge over~$V$ we find
    \begin{equation}\label{eq:charge_subvariety}
        \Zch_V(E)=\big(\rho_0\,\Chern_1(E)+\rk(E)\left(\rho_0 U_1+\rho_1[\omega]\right)\big).[V].
    \end{equation}
    Hence,~$\Im\left(\overline{\Zch_X(E)}\Zch_V(E)\right)$ equals
    \begin{equation}
    \begin{split}
        &\Im\left(\overline{\Zch_X(E)}\rho_0\,\Chern_1(E)\right).[V]+\rk(E)\Im\left(\overline{\Zch_X(E)}\left(\rho_0 U_1+\rho_1[\omega]\right)\right).[V]\\
        &\hspace{1cm}=\abs{\Zch_X(E)}\,2\alpha\,\Chern_1(E).[V]+\rk(E)\Im\left(\overline{\Zch_X(E)}\left(\rho_0 U_1+\rho_1[\omega]\right)\right).[V].
    \end{split}
    \end{equation}
    So,~$\Im\left(\overline{\Zch_X(E)}\Zch_V(E)\right)>0$ is equivalent to
    \begin{equation}
        2\alpha\,\Chern_1(E).[V]+\frac{\rk(E)}{\abs{\Zch_X(E)}}\Im\left(\overline{\Zch_X(E)}\left(\rho_0 U_1+\rho_1[\omega]\right)\right).[V]>0.
        \qedhere
    \end{equation}
\end{proof}

\begin{rmk}\label{rmk:positivity_needed_stability}
    Let~$E$ be any rank~$2$ vector bundle on the K\"ahler surface~$X$, and fix a metric~$h_0\in\Hermmetric(E)$. It is easy to check, as we are working with~$(1,1)$-forms on a surface, that~$h_0$ is a solution of the following equation for~$h\in\Hermmetric(E)$
    \begin{equation}\label{eq:characteristic_Z}
        \curvform(h)^2-(\Tr\curvform(h_0))\wedge\curvform(h)+\det\curvform(h_0)\otimes\id_E=0,
    \end{equation}
    as every matrix is a zero of its characteristic polynomial. If we choose a polynomial central charge~$\Zch$ such that
    \begin{equation}\label{eq:characteristic_Z_coeff}
    \begin{dcases}
        2\,\Im\left(\overline{\Zch_X(E)}\left(\rho_0u_1+\rho_1\omega\right)\right)=-\Im\left(\overline{\Zch_X(E)}\rho_0\right)\Tr\curvform(h_0)\\
        2\,\Im\left(\overline{\Zch_X(E)}\left(\rho_0u_2+\rho_1\omega\wedge u_1+\rho_2\omega^2\right)\right)=\Im\left(\overline{\Zch_X(E)}\rho_0\right)\det\curvform(h_0),
    \end{dcases}
    \end{equation}
    then~\eqref{eq:characteristic_Z} is equivalent to the~$\Zch$-critical equation (c.f.~\eqref{eq:Zcritical_surface} and~\eqref{eq:abc_coeff}), and~$h_0$ is a~$\Zch$-critical metric. However it may very well be that~$E$ is not~$\Zch$-stable, for example if~$E$ is not simple, see Lemma~\ref{lemma:stable_simple}. This apparent contradiction with Theorem~\ref{thm:Zcrit_stability} is solved by noting that for this central charge, the~$\Zch$-critical metric~$h_0$ can never be~$\Zch$-positive: indeed we saw in the proof of Lemma~\ref{lemma:subsol_positiveclass} that if~$h$ is a~$\Zch$-positive metric we should have (with the notation of~\eqref{eq:abc_coeff})
    \begin{equation}
        2\alpha\Tr\curvform(h)+\rk(E)\beta>0,
    \end{equation}
    but~\eqref{eq:characteristic_Z} and~\eqref{eq:characteristic_Z_coeff} imply~$2\alpha\Tr\curvform(h_0)+\rk(E)\beta=0$, since~$\beta/2\alpha=-\Tr\curvform(h_0)/2$. This shows that it is not possible to deduce~$\Zch$-stability from the existence of a~$\Zch$-critical metric, without some positivity condition on the metric, see also Example~\ref{ex:not_positive_P2}.    
\end{rmk}

Putting together Proposition~\ref{prop:stability_comparison} and the equivalence of~$\Zch$-positivity and (twisted) Monge-Ampère positivity explained in Section~\ref{sec:subsolutions} with Pingali's stability result (Theorem~\ref{thm:Pingali_stab}) will prove Theorem~\ref{thm:Zcrit_stability} in the case~$\alpha>0$, except for the possible non-rationality of~$\beta$ and the fact that the statement of Theorem~\ref{thm:Pingali_stab} in~\cite{Pingali_vbMA} requires that the form~$\eta$ in~\eqref{eq:vbMA} is a positive top-degree form. We can translate this through the correspondence in Lemma~\ref{lemma:equation_equivalence} to obtain a condition for a general~$\Zch$-critical metric: with the notation of~\eqref{eq:abc_coeff},
\begin{equation}\label{eq:volumeform_hyp}
    \eta>0\iff\beta^2-4\alpha\gamma>0.
\end{equation}
We call~\eqref{eq:volumeform_hyp} the \emph{volume form hypothesis}. This hypothesis is in fact not necessary for Theorem~\ref{thm:Pingali_stab} to hold, as we show in Section~\ref{sec:positivity_stability}.

An algebraic analogue of Proposition~\ref{prop:Bogo-analytic} can be given in the projective setting. Its proof does not require the use of~$\Zch$-critical metrics, nor the volume form hypothesis, thus providing additional evidence for Conjecture~\ref{conj:Zpositive}. Note that this result also implies the conjecture in~\cite[page~$4$]{Pingali_vbMA}.
\begin{prop}
    Assume that~$X$ is a smooth projective surface,~$\Zch$ is a polynomial central charge, and~$E$ is a rank~$2$ vector bundle which is~$\Zch$-positive, satisfies~\eqref{eq:Z_alphapositive}, and is~$\Zch$-stable. Then~$E$ is Mumford stable with respect to the K\"ahler class~$c_1(E)+2\theta$ and the Bogomolov inequality~\eqref{eq:Bogo} holds.
\end{prop}
\begin{proof}
Denoting~$\theta'=[\beta/\alpha]=2\theta$, the positivity assumption implies (as in the proof of Lemma~\ref{lemma:subsol_positiveclass}) that
$A:=c_1(E)+\theta'>0$. Let's check that 
$E$ is Mumford stable with respect to the positive class~$A$. If it is not the case, as~$E$ has rank~$2$, there exists a subline bundle~$L_1$ of~$E$ with torsion free quotient that destabilizes (in the sense of Mumford), see~\cite[Chapter 4, p.87]{Friedman}. The holomorphic bundle~$E$ sits in an exact sequence
\begin{equation}
    0 \longrightarrow L_1 \longrightarrow E \longrightarrow L_2 \otimes \mathcal{I}_Y\longrightarrow 0
\end{equation}
where~$Y\subset X$ is a closed subscheme of dimension~$0$. Actually, as we are working over a surface,~$E/L_1$ is isomorphic to~$L_2\otimes \mathcal{I}_Y$ where~$L_2$ is invertible and~$Y$ is a locally complete intersection of~$X$ of dimension~$0$, see~\cite[Chapter 2]{Friedman}. 

If~$\ell(Y)\geq 0$ denotes the length of~$Y$, then we have the relations~$\Chern_1(E)=\Chern_1(L_1)+\Chern_1(L_2)$ and ~$\chern_2(E)=\Chern_1(L_1)\Chern_1(L_2)+\ell(Y)$ from the fact that~$\chern_2(\mathcal{I}_Y)=\ell(Y)$.
Using above relations, the fact that~$L_1$ destabilizes~$E$ writes as 
\begin{equation}
    (\Chern_1(L_2)-\Chern_1(L_1)).(\Chern_1(L_1)+\Chern_1(L_2)+\theta')< 0,
\end{equation}
which is equivalent to
\begin{equation} \label{eq:ineq1}
    \Chern_1(L_2)^2-\Chern_1(L_1)^2+(\Chern_1(L_2)-\Chern_1(L_1))\theta'<0.
\end{equation}

Now, by Proposition~\ref{prop:stability_comparison} and Remark~\ref{rmk:strong_Zpositive}, the bundle~$E$ is Monge-Amp\`ere stable and the Monge-Amp\`ere slope inequality applied to~$L_1$ and~$E$ with respect to~$\theta$ provides the inequality
\begin{equation}\label{eq:devissageZstab}
    \Chern_1(L_2)^2-\Chern_1(L_1)^2+\left(\Chern_1(L_2)-\Chern_1(L_1)\right)\theta'>2\ell(Y).
\end{equation}
With inequality~\eqref{eq:ineq1}, this provides the expected contradiction. Thus~$E$ is Mumford stable with respect to~$A$
and we can apply ~\cite[Theorem~$1$]{Reid} (a careful reading of Reid's proof shows that it holds for any integral ample class and we can approximate~$A$ by rational classes).
\end{proof}

\begin{exm}\label{exm:CP2blowup}
It is easy to construct rank 2 bundles over a surface which are Monge-Ampère stable,~$\Zch$-stable,~$\Zch$-positive, and~$\Zch$-positive for quotients, but are not Mumford stable with respect to certain polarizations. For a positive integer~$r$, let's consider a non-split extension~$$0\to \mathcal{L}_r\to E\to \sheaf_X\to 0$$ over~$X$, the blow-up of~$\PP^2$ at one point. Let~$H$ be the hyperplane section and~$E_1$ is the exceptional divisor. We consider~$[\omega]=pH-qE_1$ with positive integers~$p,q$ such that~$p>q$.
Now we choose~$\mathcal{L}_r=r(qH-pE_1)$,~$r\in \mathbb{N}^*$ so that
\begin{equation}\label{eq:ch1}
\Chern_1(\mathcal{L}_r).[\omega]=0=\Chern_1(E).[\omega],
\end{equation}
and this implies that~$E$ is strictly Mumford semi-stable with respect to the positive class~$\omega$.
Moreover,
\begin{equation}\label{eq:ch2}
\Chern_2(\mathcal{L}_r)=\frac{r^2}{2}(q^2-p^2)<0.
\end{equation}
Thus, for large~$r\gg 1$ the extension is non trivial as~$h^1(X,\mathcal{L}_r)\neq 0$. This can be deduced from the Hirzebruch-Riemann-Roch formula, since
$h^0(X,\mathcal{L}_r)-h^1(X,\mathcal{L}_r)+h^2(X,\mathcal{L}_r)=\frac{r^2}{2}(q^2-p^2)+O(r)$ is negative for large~$r$.

 Now, let's consider a polynomial central charge~$\Zch$ for which~$U_1$ is proportional to~$[\omega]$. From~\eqref{eq:abc_coeff}, a computation shows  that 
$$\vert Z_X(E)\vert\alpha=U_1.[\omega]\Im(\rho_0\overline{\rho_1})+
[\omega]^2\Im(\rho_0\overline{\rho_2})$$ which is independent of~$r$ and will be positive if~$U_1$ is well-chosen. Furthermore, 
\begin{align*}\vert Z_X(E)\vert\beta=&\ch_2(E)\Im(\overline{\rho_0}\rho_1)\omega + 2\Im((\rho_0 u_1+\rho_1\omega)(\overline{\rho_0}u_2+\overline{\rho_1}u_1.[\omega]+\overline{\rho_2}[\omega]^2))\\
=&r^2\left(\frac{q^2-p^2}{2}\right)\Im(\overline{\rho_0}\rho_1)\omega+ 2 \Im(\rho_0\overline{\rho_1}u_1.[\omega] + \rho_0\overline{\rho_2}[\omega^2])u_1
+ 2\Im(\rho_1\overline{\rho_2})[\omega^2]\omega\\
&+ 2\Im(\overline{\rho_0}\rho_1)U_2\omega.
\end{align*}
Since~$\Im(\rho_1\overline{\rho_0})<0$,  we can choose~$U_2$ proportional to~$[\omega]^2$ and independent of~$r$ such that~$\beta=\kappa \omega$ with~$\kappa>0$ possibly very large.

We want to show that~$E$ is Monge-Ampère stable with respect to~$\vartheta=[{\beta}/{2\alpha}]$. Firstly, we note that 
$\ch_2(\mathcal{L}_r)=\mu_{MA,\vartheta}(\mathcal{L}_r)<\mu_{MA,\vartheta}(\mathcal{O}_X)$ using~\eqref{eq:ch1} and~\eqref{eq:ch2}. If not, there will be a rank 1 torsion free susbheaf~$\mathcal{L}'$ that will destabilize~$E$ with respect to the Monge-Ampère slope~$\mu_{MA,\vartheta}$. In the above exact sequence, let us denote the maps~$\theta_1:E\to \mathcal{O}_X$ and~$\theta_2:\mathcal{L}'\to E$. Since~$\mathcal{L}'$ has rank one, it is Monge-Ampère stable and so the map~$\theta_1\circ \theta_2$ is either trivial or an isomorphism. In the first case, we get a new map~$\mathcal{L}'\to \mathcal{L}$ and since~$\mathcal{L}$ is Monge-Ampère stable, this implies~$\mu_{MA,\vartheta}(\mathcal{L}')< \mu_{MA,\vartheta}(\mathcal{L})$. But this contradicts the assumption that~$\mathcal{L}'$ destabilizes~$E$. In the second case, the isomorphisms gives~$\mathcal{L}'=\mathcal{O}_X$ and the extension splits which contradicts our construction. Eventually, this shows that~$E$ is Monge-Ampère stable. Consequently, under the above assumptions,~$E$ is~$\Zch$-stable by Proposition~\ref{prop:stability_comparison}.

Moreover, it is easy to check directly that~$E$ is~$\Zch$-positive and~$\Zch$-positive for quotients. Actually, one has that~$\vert Z_X(E)\vert(2\alpha \chern_1(E)+[\beta] \rk(E))$ 
is positive for large~$\kappa$ and the proof of Lemma Lemma~\ref{lemma:subsol_positiveclass} shows~$\Zch$-positivity. Moreover, over any curve~$V$,~$(2\alpha\chern_1(\mathcal{O}_X)+\beta).V>0$ as~$[\omega]$  is positive. 

We investigate the existence of~$\Zch$-critical metrics for this bundle in Example~\ref{exm:suite}.  Note that the choices of~$r$ and~$U_1,U_2$   can be done effectively. Notice that if we choose~$\mathcal{L}_r=r((q+1)H-pE_1)$, the same computations provide that~$E$ enjoys similar properties but is not Mumford semistable with respect to~$[\omega]$.
\end{exm}

Polynomial central charges are additive on short exact sequences, and this implies in particular that
\begin{lemma}[\cite{DervanMcCarthySektnan}]\label{lemma:stable_simple}
    If~$E\to X$ is a~$\Zch$-stable vector bundle for some polynomial central charge~$\Zch$, then~$E$ is simple.
\end{lemma}
\begin{proof}
    As the Chern character~$\Chern(E)$ is additive on short exact sequences, if~$E$ is~$\Zch$-stable any vector bundle morphism~$\phi:E\to E$ is either~$0$ or an isomorphism. Indeed, considering the kernel and range of~$\phi$ we obtain
    \begin{equation}
        \Im\left(\frac{\Zch_X(\ker\phi)}{\Zch_X(E)}\right)+\Im\left(\frac{\Zch_X(\range\phi)}{\Zch_X(E)}\right)=\Im\left(\frac{\Zch_X(E)}{\Zch_X(E)}\right)=0.
    \end{equation}
    However~$\ker\phi$ and~$\range\phi$ are subsheaves of~$E$, and if~$0<\rk(\ker\phi)<\rk(E)$ the~$\Zch$-stability of~$E$ would imply
    \begin{equation}
        \Im\left(\frac{\Zch_X(\ker\phi)}{\Zch_X(E)}\right)<0 \text{ or } \Im\left(\frac{\Zch_X(\range\phi)}{\Zch_X(E)}\right)<0,
    \end{equation}
    which is a contradiction. But then~$\phi$ must be the multiplication by a scalar: let~$\lambda$ be any eigenvalue of~$\phi$ on a fibre~$E_x$, and apply the previous result to the morphism~$\hat{\phi}=\phi-\lambda\id_E$. As it is not invertible, it must be identically zero.
\end{proof}

\subsection{Consequences of Monge-Ampère positivity}\label{sec:positivity_stability}

This Section rephrases and expands the proof of~\cite[Lemma~$3.1$]{Pingali_vbMA}. We write the details both for completeness and to highlight the importance of the Monge-Ampère positivity of a solution of~\eqref{eq:vbMA_twisted}, which is stronger than the positivity assumption in~\cite{Pingali_vbMA} but is more natural from our point of view. We also highlight the fact that it is not necessary to assume the volume form hypothesis, and we take into account a non-vanishing twist~$\beta/2\alpha$ in~\eqref{eq:vbMA_twisted}. Most of the computation will be done for vector bundles of arbitrary rank, as we hope that our approach can be used to establish Theorem~\ref{thm:Zcrit_stability} for higher rank bundles.

Consider a short exact sequence of vector bundles over the compact K\"ahler surface~$X$
\begin{equation}\label{eq:ses}
0\to S\to E\to Q\to 0.
\end{equation}
Fix a Hermitian metric~$h$ on~$E$, and let~$A\in\Alt^{0,1}(\Hom(Q,S))$ be the second fundamental form of~\eqref{eq:ses}. Then the curvature forms of~$h$, its restriction~$h_S$ on~$S$ and the induced metric~$h_Q$ on~$Q$ are related by
\begin{equation}\label{eq:curvature_secondfundform}
    \curvform_{\alpha,\beta}(h)=\begin{pmatrix}
        \curvform_{S,\alpha,\beta}-\frac{\I}{2\pi}A\wedge A^* & \frac{\I}{2\pi}D'A \\ -\frac{\I}{2\pi}D''A^* & \curvform_{Q,\alpha,\beta}-\frac{\I}{2\pi}A^*\wedge A
    \end{pmatrix}
\end{equation}
where we denote by~$\curvform_S$ and~$\curvform_Q$ the curvature forms of~$h_S$ and~$h_Q$ respectively, and following the previous notation we set~$\curvform_{S,\alpha,\beta}=\curvform_S+\frac{\beta}{2\alpha}\otimes\id_S$. Note also that~\eqref{eq:curvature_secondfundform} depends on the decomposition~$\beta\otimes\id_E=\beta\otimes\id_S\oplus \beta\otimes\id_Q$.

Squaring~\eqref{eq:curvature_secondfundform} we obtain
\begin{equation}
\begin{split}
    &\curvform_{\alpha,\beta}(h)^2=\\
    =&\begin{pmatrix}
        \left(\curvform_{S,\beta}-\frac{\I}{2\pi}A\wedge A^*\right)^2-\left(\frac{\I}{2\pi}\right)^2D'A\wedge D''A^* & \hspace{-2cm}(\dots) \\ (\dots) & \hspace{-2cm}\left(\curvform_{Q,\alpha,\beta}-\frac{\I}{2\pi} A^*\wedge A\right)^2-\left(\frac{\I}{2\pi}\right)^2D''A^*\wedge D'A
    \end{pmatrix}.
\end{split}
\end{equation}
Assume that~$h$ is a solution of~$\curvform_{\alpha,\beta}(h)^2=\eta\otimes\id_E$; then, we must have
\begin{equation}
    \begin{dcases}
        \left(\curvform_{S,\alpha,\beta}-\frac{\I}{2\pi}A\wedge A^*\right)^2-\left(\frac{\I}{2\pi}\right)^2D'A\wedge D''A^*=\eta\otimes\id_S\\
        \left(\curvform_{Q,\alpha,\beta}-\frac{\I}{2\pi}A^*\wedge A\right)^2-\left(\frac{\I}{2\pi}\right)^2D''A^*\wedge D'A=\eta\otimes\id_Q.
    \end{dcases}
\end{equation}
Taking the traces and integrating, we find two equations relating~$\Chern_2(S)$,~$\Chern_2(Q)$, and~$A$ to~$\rk(S)$,~$\rk(Q)$, and~$\vartheta=[\beta/2\alpha]$, with some spurious terms. More precisely, for~$S$ we get
\begin{equation}\label{eq:vbMA_inteq_S}
\begin{split}
    \rk(S)\,[\eta].X=&2\,\rk(S)\mu_{MA,\vartheta}(S)+\rk(S)\vartheta^2-\frac{2\I}{2\pi}\int_X\Tr\left(\curvform_{S,\alpha,\beta}\wedge A\wedge A^*\right)\\
    &+\left(\frac{\I}{2\pi}\right)^2\int_X\Tr\left(\left(A\wedge A^*\right)^2\right)-\left(\frac{\I}{2\pi}\right)^2\int_X\Tr\left(D'A\wedge D''A^*\right)
\end{split}
\end{equation}
while for~$Q$
\begin{equation}\label{eq:vbMA_inteq_Q}
\begin{split}
    \rk(Q)\,[\eta].X=&2\,\rk(Q)\mu_{MA,\vartheta}(Q)+\rk(Q)\vartheta^2-\frac{2\I}{2\pi}\int_X\Tr\left(\curvform_{Q,\alpha,\beta}\wedge A^*\wedge A\right)\\
    &+\left(\frac{\I}{2\pi}\right)^2\int_X\Tr\left(\left(A^*\wedge A\right)^2\right)-\left(\frac{\I}{2\pi}\right)^2\int_X\Tr\left(D''A^*\wedge D'A\right).
\end{split}
\end{equation}
Note now that~\eqref{eq:vbMA} implies
\begin{equation}
    2\,\rk(E)\mu_{MA,\vartheta}(E)+\rk(E)\vartheta^2=\rk(E)\,[\eta].X=(\rk(S)+\rk(Q))[\eta].X
\end{equation}
and from the exact sequence we get
\begin{equation}
\begin{split}
    &2\,\rk(E)\mu_{MA,\vartheta}(E)+\rk(E)\vartheta^2=\\
    =&2\,\rk(S)\mu_{MA,\vartheta}(S)+\rk(S)\vartheta^2+2\,\rk(Q)\mu_{MA,\vartheta}(Q)+\rk(Q)\vartheta^2.
\end{split}
\end{equation}
Hence, adding~\eqref{eq:vbMA_inteq_S} and~\eqref{eq:vbMA_inteq_Q} we obtain
\begin{equation}\label{eq:vbMA_stability_identity1}
    \begin{split}
        &-\frac{2\I}{2\pi}\int_X\Tr\left(\curvform_{S,\alpha,\beta}\wedge A\wedge A^*\right)-\frac{2\I}{2\pi}\int_X\Tr\left(\curvform_{Q,\alpha,\beta}\wedge A^*\wedge A\right)\\
    &+\left(\frac{\I}{2\pi}\right)^2\int_X\Tr\left(\left(A\wedge A^*\right)^2\right)+\left(\frac{\I}{2\pi}\right)^2\int_X\Tr\left(\left(A^*\wedge A\right)^2\right)\\
    &-\left(\frac{\I}{2\pi}\right)^2\int_X\Tr\left(D'A\wedge D''A^*\right)-\left(\frac{\I}{2\pi}\right)^2\int_X\Tr\left(D''A^*\wedge D'A\right)=0.
    \end{split}
\end{equation}
\begin{lemma}\label{lemma:trace_indentities}
With the previous notation, we have
    \begin{equation}
        \Tr\left(\left(A^*\wedge A\right)^2\right)+\Tr\left(\left(A\wedge A^*\right)^2\right)=0,
    \end{equation}
    \begin{equation}
        \Tr\left(D'A\wedge D''A^*\right)=\Tr\left(D''A^*\wedge D'A\right).
    \end{equation}
\end{lemma}
\begin{proof}
This is a consequence of the general fact that for any~$C\in\Alt^k(\End(E))$ and~$B\in\Alt^l(\End(E))$,~$\Tr(C\wedge B)=(-1)^{kl}\Tr(B\wedge C)$. More explicitly, if~$B$ and~$C$ are~$\End(E)$-valued~$1$-forms, we have
    \begin{equation}
    \begin{split}
        \Tr&(C\wedge B\wedge C\wedge B)=\Tr\left(C_iB_jC_kB_l\right)\dd x^i\wedge\dd x^j\wedge\dd x^k\wedge\dd x^l=\\
        &=-\Tr\left(B_jC_kB_lC_i\right)\dd x^j\wedge\dd x^k\wedge\dd x^l\wedge\dd x^i=-\Tr\left(B\wedge C\wedge B\wedge C\right).
    \end{split}
    \end{equation}
    This proves the first identity, identifying~$A$ with the~$\End(E)$-valued differential form
    \begin{equation}
        \begin{pmatrix}
            0 & A\\ 0 &0
        \end{pmatrix}
    \end{equation}
    in the matrix-block notation corresponding to the decomposition~$E=S+Q$ as smooth vector bundles. The other identity is proved in the same way.
\end{proof}

Hence~\eqref{eq:vbMA_stability_identity1} becomes
\begin{equation}
\begin{split}
    \frac{\I}{2\pi}\int_X\Tr\left(\curvform_{S,\alpha,\beta}\wedge A\wedge A^*\right)+\frac{\I}{2\pi}\int_X\Tr\left(\curvform_{Q,\alpha,\beta}\wedge A^*\wedge A\right)=&\\
    =-\left(\frac{\I}{2\pi}\right)^2\int_X&\Tr\big(D''A^*\wedge D'A\big).
\end{split}
\end{equation}
Substitute this expression for~$\int\Tr\left(D''A^*\wedge D'A\right)$ in~\eqref{eq:vbMA_inteq_S}
\begin{equation}\label{eq:vbMA_inteq_S_simplified}
\begin{split}
    \rk(S)\,[\eta].X=&2\,\rk(S)\mu_{MA,\vartheta}(S)+\rk(S)\vartheta^2-\frac{\I}{2\pi}\int_X\Tr\left(\curvform_{S,\alpha,\beta}\wedge A\wedge A^*\right)+\\
    &+\frac{\I}{2\pi}\int_X\Tr\left(\curvform_{Q,\alpha,\beta}\wedge A^*\wedge A\right)+\left(\frac{\I}{2\pi}\right)^2\int_X\Tr\left(\left(A\wedge A^*\right)^2\right)
\end{split}
\end{equation}
As~$2\,\mu_{MA,\vartheta}(E)+\vartheta^2=[\eta].X$, the inequality of Monge-Amp\`ere slopes
\begin{equation}
    \mu_{MA,\vartheta}(S)<\mu_{MA,\vartheta}(E)
\end{equation}
is equivalent to
\begin{equation}\label{eq:vbMA_stability_keycondition}
    \frac{\I}{2\pi}\int_X\Tr\left(\curvform_{Q,\alpha,\beta}\wedge A^*\wedge A\right)-\frac{\I}{2\pi}\int_X\Tr\left(\curvform_{S,\alpha,\beta}\wedge A\wedge A^*\right)>\left(\frac{\I}{2\pi}\right)^2\int_X\Tr\left(\left( A^*\wedge A\right)^2\right).
\end{equation}
\begin{lemma}\label{lemma:subsol1}
    Assume that~$h\in\Hermmetric(E)$ is Monge-Ampère positive. Then,
    \begin{equation}
        \frac{\I}{2\pi}\int_X\Tr\left(\curvform_{Q,\alpha,\beta}\wedge A^*\wedge A\right)-\frac{\I}{2\pi}\int_X\Tr\left(\curvform_{S,\alpha,\beta}\wedge A\wedge A^*\right)\geq 2\left(\frac{\I}{2\pi}\right)^2\int_X\Tr\left((A^*\wedge A)^2\right)
    \end{equation}
    with equality if and only if~$A=0$. If instead~$h\in\Hermmetric(E)$ is Monge-Ampère negative then
    \begin{equation}
        \frac{\I}{2\pi}\int_X\Tr\left(\curvform_{Q,\alpha,\beta}\wedge A^*\wedge A\right)-\frac{\I}{2\pi}\int_X\Tr\left(\curvform_{S,\alpha,\beta}\wedge A\wedge A^*\right)\leq 2\left(\frac{\I}{2\pi}\right)^2\int_X\Tr\left((A^*\wedge A)^2\right)
    \end{equation}
    with equality if and only if~$A=0$.
\end{lemma}
Note that this inequality (in the positive case, say) is quite similar to~\eqref{eq:vbMA_stability_keycondition}, the only difference is a factor of~$2$. Still, Lemma~\ref{lemma:subsol1} implies~\eqref{eq:vbMA_stability_keycondition} if~$(\I)^2\Tr\left((A^*\wedge A)^2\right)>0$, but this last inequality might not be true in general.
\begin{proof}[Proof of Lemma~\ref{lemma:subsol1}]
    The two cases are symmetrical, so we focus on the positive one. Choose~$\xi$ to be~$\begin{pmatrix}0 & A\\ 0 &0\end{pmatrix}$. From~\eqref{eq:curvature_secondfundform} we have
    \begin{equation}
    \begin{split}
        \xi^*\wedge\xi\wedge\curvform_{\alpha,\beta}(h)=&
        \begin{pmatrix}
        0 & 0 \\ A^* &0
        \end{pmatrix}
        \wedge
        \begin{pmatrix}
        0 & A\\ 0 &0
        \end{pmatrix}
        \wedge
        \begin{pmatrix}
        \curvform_{S,\alpha,\beta}-\frac{\I}{2\pi}A\wedge A^* & \frac{\I}{2\pi}D'A \\ -\frac{\I}{2\pi}D''A^* & \curvform_{Q,\alpha,\beta}-\frac{\I}{2\pi}A^*\wedge A
        \end{pmatrix}=\\
        =&
        \begin{pmatrix}
        0 & 0 \\0 & A^*\wedge A
        \end{pmatrix}
        \wedge
        \begin{pmatrix}
        \curvform_{S,\alpha,\beta}-\frac{\I}{2\pi}A\wedge A^* & \frac{\I}{2\pi}D'A \\ -\frac{\I}{2\pi}D''A^* & \curvform_{Q,\alpha,\beta}-\frac{\I}{2\pi}A^*\wedge A
        \end{pmatrix}=\\
        =&\begin{pmatrix}
        0 & 0 \\ (\dots) & A^*\wedge A\wedge\curvform_{Q,\alpha,\beta}-\frac{\I}{2\pi}(A^*\wedge A)^2
        \end{pmatrix}
    \end{split}
    \end{equation}
    while for the other factor we get
    \begin{equation}
    \begin{split}
        \xi^*\wedge\curvform_{\alpha,\beta}(h)\wedge\xi=&
        \begin{pmatrix}
        0 & 0 \\ A^* &0
        \end{pmatrix}
        \wedge
        \begin{pmatrix}
        \curvform_{S,\alpha,\beta}-\frac{\I}{2\pi}A\wedge A^* & \frac{\I}{2\pi}D'A \\ -\frac{\I}{2\pi}D''A^* & \curvform_{Q,\alpha,\beta}-\frac{\I}{2\pi}A^*\wedge A
        \end{pmatrix}
        \wedge
        \begin{pmatrix}
        0 & A\\ 0 &0
        \end{pmatrix}=\\
        =&
        \begin{pmatrix}
        0& 0\\
        A^*\wedge\curvform_{S,\alpha,\beta}-\frac{\I}{2\pi}A^*\wedge A\wedge A^* & (\dots)
        \end{pmatrix}
        \wedge
        \begin{pmatrix}
        0 & A\\ 0 &0
        \end{pmatrix}=\\
        =&
        \begin{pmatrix}
        0& 0\\
        0 & A^*\wedge\curvform_{S,\alpha,\beta}\wedge A-\frac{\I}{2\pi}(A^*\wedge A)^2 
        \end{pmatrix}
    \end{split}
    \end{equation}
    Taking the trace, we see that the subsolution condition implies
    \begin{equation}
        \I\Tr\left(A^*\wedge A\wedge\curvform_{Q,\alpha,\beta}-\frac{\I}{2\pi}(A^*\wedge A)^2\right)+\I\Tr\left(A^*\wedge\curvform_{S,\alpha,\beta}\wedge A-\frac{\I}{2\pi}(A^*\wedge A)^2\right)\geq 0
    \end{equation}
    with equality if and only if~$A=0$. We rewrite this as
    \begin{equation}
        \frac{\I}{2\pi}\Tr\left(\curvform_{Q,\alpha,\beta}\wedge A^*\wedge A\right)-\frac{\I}{2\pi}\Tr\left(\curvform_{S,\alpha,\beta}\wedge A\wedge A^*\right)-2\left(\frac{\I}{2\pi}\right)^2\Tr\left((A^*\wedge A)^2\right)\geq 0
    \end{equation}
    which gives the thesis.
\end{proof}
As a consequence, we obtain a first extension of Theorem~\ref{thm:Pingali_stab}.
\begin{cor}\label{cor:rank2_MAstability}
    Let~$E$ be a simple vector bundle over a K\"ahler surface. If~$h\in\Hermmetric(E)$ is twisted Monge-Ampère positive and solves the twisted Monge-Ampère equation \eqref{eq:vbMA}, then for every sub-bundle~$S\subset E$ of corank~$1$
    \begin{equation}\label{eq:MAstable_subbundle}
        \mu_{MA,\vartheta}(S)\leq\mu_{MA,\vartheta}(E),
    \end{equation}
    with equality if and only if~$S$ splits off $E$ as a direct summand. If moreover $\rk(E)=2$, then $E$ is Monge-Ampère stable.
\end{cor}
\begin{proof}
    Assuming that the second fundamental form~$A$ satisfies~$\I^2\Tr((A^*\wedge A)^2)\geq 0$, the first inequality in Lemma~\ref{lemma:subsol1} is stronger than~\eqref{eq:vbMA_stability_keycondition}, so Lemma~\ref{lemma:subsol1} will imply the first part of the Lemma.
    
    We prove that~$\I^2\Tr((A^*\wedge A)^2)\geq 0$ if~$\rk(S)=\rk(E)-1$. Note first that in this case~$A^*\wedge A$ is just a~$(1,1)$-form. In an orthonormal local frame $(f_1,\dots,f_r)$ of $E$ such that $\{f_1,\dots,f_{r-1}\}$ spans $S$, we can write~$A=A_{\bar{a}}\dd\bar{z}^a$ for matrices~$A_{\bar{a}}$, and~$A^*=A_{\bar{a}}^*\dd z^a$; hence,
    \begin{equation}
        \left(\I\,A^*\wedge A\right)^2=2\left(A_{\bar{1}}^*A_{\bar{1}}A_{\bar{2}}^*A_{\bar{2}}-A_{\bar{1}}^*A_{\bar{2}}A_{\bar{2}}A_{\bar{1}}^*\right)\I\dd z^1\wedge\dd\bar{z}^1\wedge\I\dd z^2\wedge\dd\bar{z}^2,
    \end{equation}
    so it will be sufficient to show that~$A_{\bar{1}}^*A_{\bar{1}}A_{\bar{2}}^*A_{\bar{2}}-A_{\bar{1}}^*A_{\bar{2}}A_{\bar{2}}A_{\bar{1}}^*\geq 0$. If we write~$A_{\bar{1}}=(x^1,\dots,x^s)^\transpose$ and~$A_{\bar{2}}=(y^1,\dots,y^s)^\transpose$ we get
    \begin{equation}
        \sum_{i,j}\abs{x^i}^2\abs{y^j}^2-\sum_{i,j}\bar{x}^iy^ix^j\bar{y}^j=\sum_{i\not=j}\abs{x^i}^2\abs{y^j}^2-\sum_{i\not=j}\bar{x}^iy^ix^j\bar{y}^j=\sum_{i\not=j}\abs{\bar{x}^iy^i-x^j\bar{y}^j}^2\geq 0.
    \end{equation}
    This together with Lemma \ref{lemma:subsol1} implies \eqref{eq:vbMA_stability_keycondition}, unless $A=0$; but in that case $E$ splits as a direct sum of $S$ and its orthogonal complement, against our assumption.
    
    For a rank $2$ bundle, the inequality of Monge-Ampère slopes for general (co-)rank $1$ subsheaves follows as in the proof of~\cite[Proposition~$3.1$]{Pingali_vbMA}.
\end{proof}

\begin{proof}[Proof of Theorem~\ref{thm:Zcrit_stability}]
    Assume that~$h\in\Hermmetric(E)$ is~$\Zch$-positive and~$\Zch$-critical, and assume that~$S\subset E$ is a sub-bundle. By Proposition~\ref{prop:stability_comparison}, we want to establish
    \begin{equation}
        \alpha\left(\mu_{MA,\vartheta}(S)-\mu_{MA,\vartheta}(E)\right)<0.
    \end{equation}
    First assume that~$\alpha>0$, so that the inequality amounts to~$\mu_{MA,\vartheta}(S)<\mu_{MA,\vartheta}(E)$, i.e. Monge-Ampère stability. But in this case~$h$ is~$\Zch$-positive if and only if it is Monge-Ampère positive, so Corollary~\ref{cor:rank2_MAstability} allows to conclude, even if~$S$ is just a saturated subsheaf of~$E$.

    If instead~$\alpha$ is negative, we must prove that~$\mu_{MA,\vartheta}(S)>\mu_{MA,\vartheta}(E)$. As~$h$ is~$\Zch$-positive and~$\alpha<0$,~$h$ is Monge-Ampère \emph{negative}, and by Lemma~\ref{lemma:subsol1} then we deduce that~$\mu_{MA,\vartheta}(S)>\mu_{MA,\vartheta}(E)$ arguing as in Corollary~\ref{cor:rank2_MAstability}. Note however that the proof of~\cite[Proposition~$3.1$]{Pingali_vbMA} does not allow us to extend this last inequality to also consider subsheaves of~$E$ when~$\alpha<0$.
\end{proof}

\begin{rmk}\label{rmk:corank_1_stability}
    Corollary \ref{cor:rank2_MAstability} also gives an obstruction to the existence of $\Zch$-critical and $\Zch$-positive metrics for bundles of arbitrary degree. If $E$ satisfies~\eqref{eq:Z_alphapositive}, $S\subset E$ is a sub-bundle of corank~$1$, and~$h\in\Hermmetric(E)$ is $\Zch$-positive and~$\Zch$-critical, then
    \begin{equation}
        \Im\left(\overline{\Zch_X(E)}\Zch_X(S)\right)\leq 0,
    \end{equation}
    with equality if and only if~$E$ splits. The same conclusion applies if instead $E$ satisfies the opposite inequality in \eqref{eq:Z_alphapositive}, i.e.~$\Im(\overline{\Zch_X(E)}\,\rho_0)$ is negative, and~$\rk(S)=1$.
\end{rmk}

\begin{proof}[Proof of Theorem~\ref{thm:Zcrit_positivity}]
    The first statement is given by Lemma~\ref{lemma:subsol_positiveclass}, which is a special case of~\cite[Theorem~$1.6$]{McCarthy_thesis}. For the second part, we actually prove a slightly stronger result that does not depend on the rank of~$E$.
    \begin{lemma}\label{lemma:Zpositivity_rank1quotient}
        Let~$\Zch$ be a polynomial central charge, and assume that~$E\to X$ is a vector bundle on the projective surface~$X$ that satisfies~\eqref{eq:Z_alphapositive}. If there exists a~$\Zch$-positive metric~$h\in\Hermmetric(E)$, then
        \begin{equation}
            \Im\left(\overline{\Zch_X(E)}\Zch_V(Q)\right)>0
        \end{equation}
        for any~$1$-dimensional analytic (irreducible and reduced) subvariety~$V\subset X$ and any torsion-free rank~$1$ quotient~$Q$ of~$E_{\restriction V}$.
    \end{lemma}    
    First, we claim that the inequality~$\Im\left(\overline{\Zch_X(E)}\Zch_V(Q)\right)>0$ is equivalent to
    \begin{equation}\label{eq:quotient_positive_V}
        \left(2\alpha\,\chern_1(Q)+[\beta]\right).V>0.
    \end{equation}
    Indeed, using~\eqref{eq:abc_coeff} we compute:    \begin{equation}\label{eq:positivity_quotient_accessory}
        \begin{split}
            \Im\left(\overline{\Zch_X(E)}\Zch_V(Q)\right)=&\Im\left(\overline{\Zch_X(E)}\left(\rho_0\,\Chern_1(Q).V+\left(\rho_0\, U_1+\rho_1\,H\right).V\right)\right)=\\
            =&\abs{\Zch_X(E)}^{-1}\left(2\alpha\,\Chern_1(Q).V+[\beta].V\right)
        \end{split}
    \end{equation}
    which is positive if and only if~\eqref{eq:quotient_positive_V} is satisfied.

    Second, we prove that under the hypotheses of Lemma~\ref{lemma:Zpositivity_rank1quotient} and assuming that~$V$ is smooth we have~\eqref{eq:quotient_positive_V}. Recall that~$E_{\restriction V}$ is an extension as in~\eqref{eq:ses},
    \begin{equation}\label{eq:ses_subvariety}
        0\rightarrow S\rightarrow E_{\restriction V}\rightarrow Q\rightarrow 0,
    \end{equation}
    with~$Q$ torsion free of rank~$1$. But torsion free sheaves over a normal variety are smooth outside a codimension~$2$ Zariski closed subset, so~$Q\to V$ is actually a line bundle since~$V$ is a smooth curve.
    
    We can find a small open neighbourhood~$U$ of any~$x$ in~$X$ such that~\eqref{eq:ses_subvariety} extends to~$U$, so on~$U$ we have the decomposition~\eqref{eq:curvature_secondfundform} for the curvature of~$h$ in terms of the curvature forms of~$h_S$,~$h_Q$, and the second fundamental form~$A\in\Alt^{0,1}(U,\Hom(Q,S))$. We know that for any~$p\in U$ and any nonzero~$\xi\in T^{0,1}X{}^*_p\otimes\End(E_p)$
    \begin{equation}
        \I\Tr\left[ \xi^*\wedge\xi\wedge\left(2\alpha\,\curvform(h)+\beta\otimes\id_E\right)\right]_{\sym}>0.
    \end{equation}
    If we choose~$\xi$ in the image of the inclusion~$\End(Q)\subset\End(E)$, from~\eqref{eq:curvature_secondfundform} we find
    \begin{equation}\label{eq:Zpos_quotient}
        \I\Tr\left[\xi^*\wedge\xi\wedge\left(2\alpha\,\curvform(h_Q)+\beta\otimes\id_Q-2\alpha\frac{\I}{2\pi}A^*\wedge A\right)\right]_{\sym}>0.
    \end{equation}
    As~$Q$ has rank~$1$, we can simplify this as
    \begin{equation}
        2\I\,\xi^*\wedge\xi\wedge\left(2\alpha\,\curvform(h_Q)+\beta\right)-4\alpha\,\I\,\xi^*\wedge\xi\wedge\frac{\I}{2\pi}A^*\wedge A>0
    \end{equation}
    and we claim that~$\alpha\,\I\,\xi^*\wedge\xi\wedge\frac{\I}{2\pi}A^*\wedge A\geq 0$ for any~$(0,1)$-form~$A$ and any~$\xi$. To show this, choose a coframe~$\zeta^1,\zeta^2$ for~$T^{1,0}{}^*X$ such that~$\xi^*=\zeta^1$, and compute
    \begin{equation}
        \I\,\xi^*\wedge\xi\wedge\,\I\,A^*\wedge A=\abs{A_2}^2\I\zeta^1\wedge\zeta^1\wedge\I\zeta^2\wedge\zeta^2.
    \end{equation}
    When~$\alpha>0$, which is guaranteed by~\eqref{eq:Z_alphapositive}, we conclude that at any point of~$V$
    \begin{equation}\label{eq:curvature_ineq_local}
        2\alpha\,\curvform(h_Q)+\beta>0,
    \end{equation}
    and in particular we obtain~\eqref{eq:quotient_positive_V}. So we have shown that Lemma~\ref{lemma:Zpositivity_rank1quotient} (and Theorem~\ref{thm:Zcrit_positivity}) holds for a smooth curve~$V\subset X$, even without the projectivity assumption.

    Lastly, assume now that~$X$ is a projective surface and~$V\subset X$ is a reduced and irreducible curve. Note first that the singular set of the torsion-free sheaf~$Q$, i.e. the set where~$Q$ is not locally free, is included in the singular set~$\sing(V)$ of~$V$, so that we still have~\eqref{eq:curvature_ineq_local} on the smooth locus of~$V$. To make sense of the expression~\eqref{eq:quotient_positive_V}, recall that~$\Chern_1(Q).V$ denotes the degree of~$Q$ over~$V$.

    We denote by~$\pi:\hat{V}\to V$ the normalization of~$V$, which is a finite morphism and an isomorphism on the smooth locus of~$V$. On~$\hat{V}$, the coherent sheaf~$\pi^*Q$ will in general have torsion, but~$Q'=\pi^*Q/\tors(\pi^*Q)$ is locally free of rank~$1$. Moreover, the projection formula implies~$\deg_V(Q)\geq\deg_{\hat{V}}(Q')$, see for example~\cite[proof of Corollary~$1.2.24$]{Lazarsfeld_positivity_1} or~\cite[proof of Lemma~$1$]{Eisenbud_deteqs_curves}, so it will be sufficient to show~\eqref{eq:quotient_positive_V} for~$Q'$ over~$V'$.

    Over the preimage of~$V\setminus\sing(V)$ in~$\hat{V}$, consider the pulled-back metric~$\pi^*(h_Q)$ on~$\pi^*Q\simeq Q'\oplus \tors(\pi^*Q)$. Recall that~$\tors(\pi^*Q)$ is supported on a proper subvariety of~$\hat{V}$, i.e. a finite number of closed points~$p_1,\dots,p_s\in\hat{V}$. We consider some small analytic neighbourhoods~$(N_i)_{i=1,\dots,s}$ of these points, and the metric~$\hat{h}$ defined by~$\pi^*(h_Q)$ on~$\hat{V}\setminus\{p_1,\dots,p_s\}$ can be smoothly extended over~$\cup_{i=1}^s N_i$ in such a way that the extension still satisfies~$2\alpha\curvform(\hat{h})+\beta\geq 0$. Since~$Q'$ has rank one, we obtain again by~\eqref{eq:curvature_ineq_local}
    \begin{equation}
        2\alpha\,\deg_{\hat{V}}(Q') +[\beta].\hat{V} \geq  \int_{\hat{V}\setminus\cup_{i=1}^s N_i}\left(2\alpha\,\curvform(\hat{h})+\beta\right)>0.
    \end{equation}
    As~$\deg_V(Q)\geq\deg_{\hat{V}}(Q')$, the assumption~\eqref{eq:Z_alphapositive} that~$\alpha>0$ then implies~\eqref{eq:quotient_positive_V}.
\end{proof}

\begin{rmk}\label{rmk:subbundles} 
    Note that if instead the opposite of~\eqref{eq:Z_alphapositive} holds, i.e.~$\alpha<0$, the same reasoning shows the inequality~\eqref{eq:quotient_positive_V} for (rank~$1$) \emph{subsheaves}~$S\subset E_{\restriction V}$ rather than quotients, starting from the analogue of the inequality~\eqref{eq:Zpos_quotient} that holds for subbundles.
\end{rmk}

\subsection{Other stability conditions}\label{sec:stabilities}

The main result of~\cite{DervanMcCarthySektnan} is the proof of a correspondence between the existence of solutions of~\eqref{eq:Zcrit} and~$\Zch$-stability in an asymptotic regime known as the \emph{large volume limit}, over K\"ahler manifolds of arbitrary dimension. More precisely, they prove (for simple bundles, but the argument applies directly to irreducible bundles) the following.
\begin{thm}[\cite{DervanMcCarthySektnan}]\label{thm:asymptotic_stability}
    Let~$Z$ be a polynomial central charge, and assume that~$E$ is irreducible and \emph{sufficiently smooth}. Then~$E$ admits a family~$\{h_k\}$ of uniformly bounded (in the~$C^2$-norm)~$\Zch_k$-critical metrics for all~$k\gg 0$ if and only if~$E$ is \emph{asymptotically~$\Zch$-stable}, i.e. it is~$\Zch_k$-stable with respect to sub-bundles for all~$k\gg 0$.
\end{thm}
A few comments about Theorem \ref{thm:asymptotic_stability} and how it relates to Conjecture~\ref{conj:Zstable} and Conjecture~\ref{conj:Zpositive} are in order. A difference between~\cite{DervanMcCarthySektnan} and the present work is that~$\Zch$-stability in~\cite{DervanMcCarthySektnan} is stated in terms of the arguments of~$\Zch_{X,k}(S)$ and~$\Zch_{X,k}(E)$, rather than the imaginary part of the ratio of these two complex numbers, for any subbundle~$S\subset E$. However, the assumption~$\Im(\rho_n)>0$ implies that, for~$k$ sufficiently large,~$\Zch_k(N)$ lies in the upper-half plane for any vector bundle~$N\to X$. Hence, the two inequalities
\begin{equation}
    \begin{gathered}
        \arg(\Zch_{X,k}(S))<\arg(\Zch_{X,k}(E)),\\
        \Im\left(\frac{\Zch_{X,k}(S)}{\Zch_{X,k}(E)}\right)<0,
    \end{gathered}
\end{equation}
are actually equivalent, for~$k\gg 0$. Also, any Hermitian metric is~$\Zch_k$-positive for~$k$ very large, so that Theorem~\ref{thm:asymptotic_stability} can be seen as a confirmation of Conjecture~\ref{conj:Zstable}. Note however that the bundle is not necessarily~$\Zch_k$-positive for all~$k\gg 0$: it is however true that~$\Im\left(\Zch_{V,k}(E_{\restriction V})\overline{\Zch_{X,k}(E)}\right)>0$ for any codimension~$1$ analytic subvariety~$V\subset X$, and in particular any bundle over a surface is asymptotically~$\Zch$-positive, see Lemma~\ref{lemma:Zpositive_asymptotic} below. This may be considered as further evidence for the fact that Conjecture~\ref{conj:Zpositive} should only be expected to hold under some \emph{critical phase condition} for the charge. This hypothetical critical phase condition will likely \emph{not} be satisfied in the large volume limit, i.e. for a very large scale parameter~$k$: for the deformed Hermitian Yang-Mills equation in rank~$1$, which is essentially the only situation where the phase condition is well-understood, the large volume limit is always \emph{sub-critical} in dimension at least~$3$.

\begin{lemma}\label{lemma:Zpositive_asymptotic}
    Fix a polynomial central charge~$\Zch$ and a bundle~$E\to X$ over a K\"ahler manifold of dimension~$n$. Assume that the coefficients~$\{\rho_1,\dots,\rho_n\}$ all lie in the same half plane. Then for any analytic subset~$V\subset X$ with~$0<\dim(V)<\dim(X)$ and any torsion-free quotient~$Q$ of~$E_{\restriction V}$ (possibly equal to~$E_{\restriction V}$), there is~$k_0$ such that for every~$k>k_0$
    \begin{equation}
        \Im\left(\frac{\Zch_{V,k}(Q)}{\Zch_{X,k}(E)}\right)>0,
    \end{equation}
    that is,~$E$ is asymptotically~$\Zch$-positive and~$\Zch$-positive for quotients. In particular, any vector bundle on a K\"ahler surface is asymptotically~$\Zch$-positive and~$\Zch$-positive for quotients for any polynomial central charge.
\end{lemma}
\begin{proof}
    Assume that~$V\subset X$ has dimension~$p>0$. By definition of a polynomial central charge, we have
    \begin{equation}
        \begin{split}
            \Zch_{X,k}(E)=&\rk(E)\,k^n\,\rho_n[\omega]^n+O(k^{n-1})\\
            \Zch_{V,k}(Q)=&\rk(Q)\,k^p\rho_p\,[\omega]^p.[V]+O(k^{p-1}).
        \end{split}
    \end{equation}
    The imaginary part of the ratio has the same sign as~$\Im\left(\overline{\Zch_{X,k}(E)}\,\Zch_{V,k}(Q)\right)$, and we can expand this as
    \begin{equation}\label{eq:asymptotic_exp_positivity}
        \Im\left(\overline{\Zch_{X,k}(E)}\,\Zch_{V,k}(Q)\right)=\rk(E)^2k^{n+p}[\omega]^n\,[\omega]^p.[V]\,\Im\left(\bar{\rho}_n\rho_p\right)+O(k^{n+p-1}).
    \end{equation}
    If~$\Im\left(\bar{\rho}_n\rho_p\right)>0$, the leading order term in this expression is positive. This condition is always true for~$p=n-1$ (by definition of a polynomial central charge) but might fail if~$V$ has higher codimension. Assume now that~$\rho_1,\dots,\rho_n$ all lie in the same half plane. Up to rotating the stability vector~$\rho$, we can assume that they all lie in the upper half plane. Since~$\Zch$ is a polynomial central charge,~$\Im\left(\bar{\rho}_j\rho_{j-1}\right)>0$ for all~$j=1,\dots,n$. In other words, using the principal value of the argument function,
    \begin{equation}
        \pi>\arg(\rho_1)>\arg(\rho_2)>\dots>\arg(\rho_{n-1})>\arg(\rho_n)>0,
    \end{equation}
    so that~$\Im\left(\rho_j\bar{\rho}_i\right)>0$ whenever~$i>j>0$.
\end{proof}

\begin{rmk}\label{rmk:positivity_noncritical}
    With the notation of Lemma~\ref{lemma:Zpositive_asymptotic}, it is clear that if the components of the stability vector~$\rho_i$ do not lie all in the same half-plane, the inequality~$\Im\left(\rho_j\bar{\rho}_i\right)>0$ might fail for some~$i>j>0$. It is however true that, for very large~$k$,
    \begin{equation}\label{eq:positivity_general}
        \Im\left(\frac{\rho_p}{\rho_n}\right)\Im\left(\frac{\Zch_{V,k}(Q)}{\Zch_{X,k}(E)}\right)>0
    \end{equation}
    provided that~$\Im\left(\rho_p\bar{\rho}_n\right)\not=0$. This shows that if the existence of~$\Zch$-positive metrics implies any inequalities between the charges of quotients of~$E_{\restriction V}$ and~$E$, in general these inequalities should depend on the codimension of~$V\subset X$ as in~\eqref{eq:positivity_general}, see also~\cite[Remark~$4.3.18$]{McCarthy_thesis}. For the rank~$1$ dHYM equation on the one-point blowup of~$\CC\PP^n$, a link between the existence of solutions and inequalities such as~\eqref{eq:positivity_general} has been investigated in~\cite{JacobSheu_BlPn}. Note however that for the dHYM charge~$\Im\left(\rho_i\bar{\rho}_j\right)=0$ if~$i$ and~$j$ have the same parity.
\end{rmk}

One could wonder if, instead of the~$\Zch$-stability condition in Conjecture~\ref{conj:Zstable}, one should also consider the~$\Zch$-stability of~$E_{\restriction V}$ for subvarieties~$V\subset X$ to characterise the existence of solutions of the~$\Zch$-critical equation, as was conjectured in~\cite{DervanMcCarthySektnan}. More precisely,~\cite[Conjecture~$1.6$]{DervanMcCarthySektnan} states that~$E$ should admit a~$\Zch$-critical connection if and only if the following condition holds:
\begin{definition}\label{def:Zstable_subvarieties}
    A bundle~$E\to X$ is \emph{$\Zch$-stable over subvarieties} if for any~$V\subset X$ and any proper short exact sequence~$0\to S\to E_{\restriction V}\to Q\to 0$ of coherent sheaves over~$V$,
    \begin{equation}\label{eq:stability_subbundles_subvarieties}
    \Im\left(\frac{Z_V(Q)}{Z_V(E_{\restriction V})}\right)>0.
\end{equation}
\end{definition}
The following result shows that, at least in our simple setting of bundles over surfaces, this condition is probably too strong.
\begin{lemma}\label{lemma:subvar_stability}
    A bundle~$E$ on a compact K\"ahler surface~$X$ is~$\Zch$-stable over subvarieties if and only if~$E$ is Mumford stable when restricted to any curve in~$X$.
\end{lemma}
\begin{proof}
Recall first that on a curve~$V\subset X$ it is sufficient to check Mumford stability over sub-bundles rather than subsheaves, so we will assume that~$S$ is a sub-bundle of~$E_{\restriction V}$. We will also stop explicitly indicating the restriction of~$E$ to~$V$.

The inequality~$\Im\left(\Zch_V(Q)\,\Zch_V(E)^{-1}\right)>0$ is equivalent to
    \begin{equation}\label{eq:imaginarypart_QV}
        \Im\left(\frac{\Zch_V(Q)}{\rk(Q)}\overline{\frac{\Zch_V(E)}{\rk(E)}}\right)>0.
    \end{equation}
By definition of the central charge, we have
\begin{equation}
    \frac{\Zch_V(E)}{\rk(E)}=\left(\rho_0\,U_1+\rho_0\frac{\Chern_1(E)}{\rk(E)}+\rho_1[\omega]\right).[V]=\rho_0\big(\mu_V(E)+U_1.[V]\big)+\rho_1\mu_V([\omega])
\end{equation}
where~$\mu_V$ denotes the slopes of sheaves on~$V$. Similarly, for~$Q$ we have
\begin{equation}
    \frac{\Zch_V(Q)}{\rk(Q)}=\rho_0(\mu_V(Q)+U_1.[V])+\rho_1\mu_V([\omega]),
\end{equation}
so we can compute the left hand-side of~\eqref{eq:imaginarypart_QV} as
\begin{equation}
\begin{split}
    &\Im\Big(
    \abs{\rho_0}^2(\mu_V(Q)+U_1.[V])(\mu_V(E)+U_1.[V])+\abs{\rho_1}^2\mu_V([\omega])^2\\
    &\quad\quad+\rho_1\overline{\rho}_0\mu_V([\omega])(\mu_V(E)+U_1.[V])+\rho_0\overline{\rho}_1\mu_V([\omega])(\mu_V(Q)+U_1.[V])\Big)\\
    &=\Im(\rho_0\overline{\rho}_1)\mu_V([\omega])\Big((\mu_V(Q)+U_1.[V])-(\mu_V(E)+U_1.[V])\Big)\\
    &=\Im(\rho_0\overline{\rho}_1)\mu_V([\omega])(\mu_V(Q)-\mu_V(E)).
\end{split}
\end{equation}
The central charge condition requires~$\Im(\rho_0\overline{\rho}_1)>0$, so~\eqref{eq:imaginarypart_QV} is equivalent to~$\mu_V(Q)>\mu_V(E)$, i.e.~$E$ must be Mumford stable when restricted to~$V$.
\end{proof}

\begin{exm}\label{ex:not_subvarstable2}
Consider an asymptotically~$\Zch$-stable bundle~$E$ over the projective surface~$X$ polarized by an ample line bundle~$L$. Assume that~$E$ is not Mumford stable with respect to~$L$. Then, from~\cite{DervanMcCarthySektnan}, it is known that~$E$ is Mumford semi-stable and there exists a coherent subsheaf~$F$ of~$E$ for which the slope satisfies the equality~$\mu(F)=\mu(E)$. Moreover, the restriction of~$E$ over a generic curve~$C$ taken in the linear system of~$L^m$ (for~$m$ large enough) is semistable by the Mehta-Ramanathan theorem for semistable bundles. But this restriction~$E_{\restriction C}$ cannot be Mumford stable as its degree can be computed using the restriction to such generic curve:~$\mu(F_{\restriction C})=\mu(E_{\restriction C})$. Consequently, Lemma~\ref{lemma:subvar_stability} shows that~$E$ can never be~$\Zch_k$-stable over subvarieties for large~$k$. 
\end{exm}

\subsection{Small variations of the charge and Mumford stability}\label{sec:alphazero}

We highlight a consequence of Theorem~\ref{thm:Zcrit_stability} that can be useful to provide new examples of~$\Zch$-critical metrics starting from known ones. First, the~$\Zch$-positivity of a~$\Zch$-critical metric implies an openness result, analogously to~\cite[Lemma~$3.4$]{Takahashi_Jeq_bundles}. 
\begin{prop}\label{prop:openness}
    If~$E\to X$ is a simple bundle, the set of polynomial central charges~$\Zch$ for which there exists a~$\Zch$-positive and~$\Zch$-critical metric~$h\in\Hermmetric(E)$ is open.
\end{prop}
\begin{proof}
    We claim that the linearisation of the~$\Zch$-critical equation at a~$\Zch$-positive and~$\Zch$-critical metric~$h$ is a self-adjoint elliptic operator, and its kernel are the holomorphic endomorphisms of~$E$.

    To prove this, we can reason as in~\cite[Lemma~$3.4$]{Takahashi_Jeq_bundles} starting from the expression of the linearisation in~\eqref{eq:derivative_Zdiff}. First of all, it is equivalent to linearise the operator with respect to a path of Hermitian metrics on~$E$ or along the action of a one-parameter subgroup of the complex gauge group, say generated by~$V\in\I\Alt^0(X,\End(E,h))$. Then the linearised operator is
    \begin{equation}
        \m{P}:V\mapsto\frac{\I}{2\pi}\left[\Im\left(\e^{-\I\vartheta_E}\Zdiff_k(D)'\right)\wedge\,(D''D'-D'D'')V\right]_{\sym}
    \end{equation}
    and it is elliptic as~$h$ is~$\Zch$-positive by~\cite[Lemma~$2.36$]{DervanMcCarthySektnan}. The Bianchi identity then shows that~$\m{P}$ is self-adjoint with respect to the pairing on endomorphisms of~$E$ given by trace and integration, see~\cite[\S$2.3.4$]{DervanMcCarthySektnan}.

    If we assume that~$\m{P}(V)=0$, then taking the trace of~$\m{P}(V)$ against~$V$ and integrating by parts we obtain
    \begin{equation}
        \left\langle D''V,D''V\right\rangle=0
    \end{equation}
    for the pairing defined in~\eqref{eq:momentmap_pairing}. So~$V$ must be a holomorphic endomorphism of~$E$, proving the claim.
    
    In our case, since~$E$ is simple the kernel of the linearisation is~$\id_E\cdot\CC$. The image of the~$\Zch$-critical operator~$h\mapsto\Zdiff$ is orthogonal to this set, so the Implicit Function Theorem will give the desired~$h'$.
\end{proof}
From Theorem~\ref{thm:Zcrit_stability} and Lemma~\ref{lemma:stable_simple} then we obtain
\begin{cor}
    Let~$\Zch$ be a polynomial central charge. If~$E$ is a rank~$2$ irreducible bundle over a K\"ahler surface that satisfies condition~\eqref{eq:Z_alphapositive}, and~$h\in\Hermmetric(E)$ is~$\Zch$-positive and~$\Zch$-critical, then for any other central charge~$\Zch'$ sufficiently close to~$\Zch$ there is~$h'\in\Hermmetric(E)$ that is~$\Zch'$-positive and~$\Zch'$-critical.
\end{cor}

We can also use Proposition~\ref{prop:openness} to construct examples of~$\Zch$-critical metrics starting from Mumford-stable bundles of arbitrary rank on a K\"ahler surface. First, note that in the case~$\alpha=0$ the~$\Zch$-critical equation~\eqref{eq:Zcritical_surface} simplifies substantially: it essentially reduces to the Hermite-Einstein equation. It is possible then to establish a version of our conjectures for bundles of arbitrary rank.
\begin{prop}\label{prop:alphazero}
    Fix a polynomial central charge~$\Zch$, and assume that~$E\to X$ is a bundle on a K\"ahler surface such that~$\alpha=0$, with the notation of~\eqref{eq:abc_coeff}. Then there is a~$\Zch$-positive and~$\Zch$-critical Hermitian metric~$h\in\Hermmetric(E)$ if and only if~$E$ is~$\Zch$-positive and~$\Zch$-stable.
\end{prop}
\begin{proof}
    When~$\alpha=0$, from~\eqref{eq:Zcritical_surface} we see that the~$\Zch$-critical equation becomes the weak Hermite-Einstein equation
    \begin{equation}\label{eq:alphazero}
        \curvform(h)\wedge\beta+\gamma\otimes\id_E=0.
    \end{equation}
    Moreover the expressions for~$\beta$ and~$\gamma$ simplify:
    \begin{equation}
        \begin{dcases}
            \beta=\Im\left(\e^{-\I\vartheta_E}\rho_1\right)\omega\\
            \gamma=\Im\left(\e^{-\I\vartheta_E}\left(\rho_1\omega\wedge u_1+\rho_2\omega^2\right)\right).
        \end{dcases}
    \end{equation}
    The key observation is that in the present situation the~$\Zch$-positivity condition for any metric~$h\in\Hermmetric(E)$ is actually independent on~$h$, and becomes just a condition on the bundle. More precisely, from~\eqref{eq:positivity_comparison} we see that~$\Zch$-positivity of a metric is equivalent to~$\beta>0$, while the proof of Lemma~\ref{lemma:subsol_positiveclass} the bundle~$E$ is~$\Zch$-stable if and only if~$[\beta]>0$. Both conditions are equivalent to~$\Im\left(\e^{-\I\vartheta_E}\rho_1\right)>0$, as~$\omega$ is K\"ahler.

    Assuming now that~$\beta$ is a K\"ahler form,~\eqref{eq:alphazero} has a solution if and only if~$E$ is Mumford stable with respect to~$\beta$; to prove our claim, it will be sufficient to show that if~$\alpha=0$~$\Zch$-stability and Mumford stability coincide. As~$\alpha=0$,~\eqref{eq:stability_comparison} gives, for any subsheaf~$S\subset E$,
    \begin{equation}
        \Im\left(\Zch_X(S)\,\overline{\Zch_X(E)}\right)=\rk(S)\,\Im\left(\overline{\Zch_X(E)}\rho_1\right)\big(-\mu_L(E)+\mu_L(S)\big).
    \end{equation}
    As~$\beta>0$ is equivalent to~$\Im\left(\overline{\Zch_X(E)}\rho_1\right)>0$, we see that~$\Im\left(\Zch_X(S)\,\overline{\Zch_X(E)}\right)<0$ if and only if~$\mu_L(S)<\mu_L(E)$.
\end{proof}

The condition~$\alpha=0$ is realised precisely when
\begin{equation}
    \Im\left(\bar{\rho}_1\rho_0U_1.[\omega]+\bar{\rho}_2\rho_0[\omega]^2+\bar{\rho}_1\rho_0\frac{\Chern_1(E).[\omega]}{\rk(E)}\right)=0
\end{equation}
which we rewrite as
\begin{equation}\label{eq:alpha=zero_condition}
    U_1.[\omega]+\frac{\Im(\bar{\rho}_2\rho_0)}{\Im(\bar{\rho}_1\rho_0)}[\omega]^2=-\frac{\Chern_1(E).[\omega]}{\rk(E)}.
\end{equation}
Under this assumption, the positivity condition~$\beta>0$, i.e.~$\Im\left(\overline{\Zch_X(E)}\rho_1\right)>0$, becomes
\begin{equation}\label{eq:beta_positive_alphazero}
    \Im(\rho_1\bar{\rho}_2)\rk(E)[\omega]^2>\Im(\rho_0\bar{\rho}_1)\big(\rk(E)U_2+U_1.\Chern_1(E)+\Chern_2(E)\big)
\end{equation}
It is relatively easy to find examples of situations where~\eqref{eq:alpha=zero_condition} and~\eqref{eq:beta_positive_alphazero} are satisfied; for example, consider a Mumford stable rank~$2$ bundle on a surface of degree zero as in Example~\ref{exm:CP2blowup}. Then~\eqref{eq:alpha=zero_condition} can be satisfied by~$\Im(\bar{\rho}_1\rho_0)U_1+\Im(\bar{\rho}_2\rho_0)[\omega]=0$, and~\eqref{eq:beta_positive_alphazero} becomes
\begin{equation}
    \Im(\rho_1\bar{\rho}_2)[\omega]^2>\Im(\rho_0\bar{\rho}_1)\left(U_2-\frac{1}{4}\chern_2(E)\right)
\end{equation}
which is satisfied for example if~$4\,U_2-\chern_2(E)<0$.

Returning to the search for examples, we can deduce that Mumford stable bundles are~$\Zch$-stable with respect to many polynomial central charges, by considering small variations around particular charges. More explicitly, for a small parameter~$t$, assume that~$\Zch^t$ is a family of polynomial central charges depending on stability vectors~$\rho^t$ and unitary charges~$U^t$, and assume that~$E\to X$ is a Mumford stable bundle (with respect to~$[\omega]$) such that
    \begin{equation}
        \begin{split}
            &\Im\left(\overline{\Zch^t(E)}\rho^t_0\right)_{\vert t=0}=0\\
            &\Im\left(\overline{\Zch^t(E)}\rho^t_1\right)_{\vert t=0}>0.
        \end{split}
    \end{equation}
Then by Proposition~\ref{prop:alphazero} we can use Proposition~\ref{prop:openness} to conclude that for every small~$t$ there is a~$\Zch_t$-critical metric on~$E$.

Note that a similar phenomenon was already observed in the large volume regime (c.f. Remark~\ref{rmk:largevolumelimit}) on higher dimensional manifolds, see~\cite[Theorem~$4.3$]{DervanMcCarthySektnan},~\cite[Theorem~$4.1$]{Pingali_vbMA}, and~\cite[Theorem~$6.5$]{Takahashi_Jeq_bundles}, while Proposition~\ref{prop:openness} holds in non-asymptotic regimes. The relevant linearised problem to establish the existence of solutions in the asymptotic regime is however essentially the same as in our case.

\section{Some examples, Gieseker stability, and Z-polystability}\label{sec:examples}

We start this Section by discussing an example showing that~$\Zch$-positivity might fail even in very simple situations. Drawing from the rank-1 theories of the J-equation and the dHYM equation, we expect that the most difficult part of establishing Conjecture~\ref{conj:Zstable} on surfaces or higher-dimensional subvarieties will be addressing when a bundle admits a~$\Zch$-positive metric. On the other hand, we expect that it should be relatively easy to check the~$\Zch$-positivity of a vector bundle, at least over a surface. For example, in the rank~$1$ case it has been shown in~\cite{Dyrefelt_Khalid_destab_curves} that one just needs to check the inequality on a finite number of curves to ensure~$\Zch$-positivity. 

\begin{exm}\label{ex:P2_charges}
    We consider polynomial central charges on~$X=\PP^2$ defined by a unitary class of the form
    \begin{equation}
        U=\e^{\lambda[\omega_{FS}]}=1+\lambda[\omega_{FS}]+\frac{\lambda^2}{2}[\omega_{FS}]^2
    \end{equation}
    for some real number~$\lambda$. We claim that the Fubini-Study Hermitian metric~$h_{FS}$ on the bundle~$E=T\PP^2\to X$ is~$\Zch$-critical for any choice of stability vector. To see this, note that the curvature of the Fubini-Study metric~$h_{FS}$ on~$E$ solves
    \begin{equation}\label{eq:FS_curvature}
    \begin{split}
        &\curvform(h_{FS})\wedge\omega_{FS}=\frac{3}{2}\omega_{FS}^2\otimes\id_E\\
        &\curvform(h_{FS})^2=\frac{3}{2}\omega_{FS}^2\otimes\id_E.
    \end{split}
    \end{equation}
    In particular,~$h_{FS}$ is a solution of the vector bundle Monge-Ampère equation (as noted in~\cite{Pingali_vbMA}), and it is also easily checked to be Monge-Ampère positive. Then the~$\Zch$-critical operator will satisfy~$\Zdiff(h_{FS})\in \CC\cdot\omega^2$ and the definition of the phase~$\e^{\I\vartheta_E}$ guarantees that for any choice of the coefficients~$\rho_i$ and~$\lambda$ we will get~$\Im(\e^{-\I\vartheta}\Zdiff(h_{FS}))=0$.

    However, for many values of these coefficients~$E$ is not~$\Zch$-positive. As an example, consider~$\rho=(1,-\I/3,-1+\I)$. We compute
    \begin{equation}
    \begin{split}
        Z_X(E)=&2\left(\rho_0\frac{\lambda^2}{2}+\rho_1\lambda+\rho_2\right)+3\left(\rho_0\lambda+\rho_1\right)+\frac{3}{2}\rho_0=\lambda^2+\lambda\left(3-\frac{2}{3}\I\right)-\frac{1}{2}+\I
    \end{split}
    \end{equation}
    and if~$H\subset X$ is a hyperplane we have instead
    \begin{equation}
        \Zch_H(E_{\restriction H})=3+2\,\lambda-\frac{2}{3}\I.
    \end{equation}
    If~$E$ were~$\Zch$-positive, by Lemma~\ref{lemma:subsol_positiveclass} we should find~$\Im\left(\overline{Z_X(E)}\Zch_H(E_{\restriction H})\right)>0$. However, we get~$\Im\left(\overline{Z_X(E)}\Zch_H(E_{\restriction H})\right)=\frac{2}{3}\left(\lambda^2-3\lambda-4\right)$,
    which is negative for~$-1<\lambda<4$.
\end{exm}

We can use this observation on~$T\PP^2$ to also show an example of a polynomial central charge~$\Zch$ and a bundle that has a~$\Zch$-positive and~$\Zch$-critical metric, but is not~$\Zch$-stable over subvarieties (c.f. Definition~\ref{def:Zstable_subvarieties}).

\begin{exm}\label{ex:not_subvarstable}
    Consider again the bundle~$E=T\PP^2$ over~$\PP^2$ and a polynomial central charge~$\Zch$ as in Example~\ref{ex:P2_charges}, i.e.~$U=\e^{\lambda\omega_{FS}}$, so that the Fubini-Study Hermitian metric on~$E$ is~$\Zch$-critical. Lemma~\ref{lemma:subvar_stability} shows that~$E$ is not~$\Zch$-stable over subvarieties: its restriction to a hyperplane~$H$ is not even semistable, as it splits as~$TH\oplus\sheaf_H(1)$, see for example~\cite[pag. 14]{OkonekSchneiderSpindler}.

    It remains to show that~$h_{FS}$ is~$\Zch$-positive for some central charge over~$\PP^2$. We choose the dHYM charge, defined by the weights~$\rho=(-\I,-1,\I/2)$, and set~$\lambda=0$ (so, the unitary class is trivial). Then, the charge is
    \begin{equation}
        Z_X(E)=-3-\frac{1}{2}\I
    \end{equation}
    and the coefficients of~\eqref{eq:abc_coeff} can be computed from the identities
    \begin{equation}
        \abs{\Zch_X(E)}\,\alpha=\frac{3}{2},\quad
        \abs{\Zch_X(E)}\,\beta=-\frac{1}{2}\omega,\quad
        \abs{\Zch_X(E)}\,\gamma=-\frac{3}{2}\omega^2.
    \end{equation}
    Note that the volume form hypothesis~\eqref{eq:volumeform_hyp} is satisfied:
    \begin{equation}
        \left(\frac{\beta}{2\alpha}\right)^2-\frac{\gamma}{\alpha}=\frac{1}{36}\omega^2+\omega^2.
    \end{equation}
    As for~$\Zch$-positivity, we should check that
    \begin{equation}
        \I\Tr\Big[\xi^*\wedge\xi\wedge\left(2\alpha\,\curvform(h_{FS})+\beta\otimes\id_E\right)+\xi^*\wedge\left(2\alpha\,\curvform(h_{FS})+\beta\otimes\id_E\right)\wedge\xi\Big]>0
    \end{equation}
    for any~$p\in X$ and any nonzero~$\xi\in T^{0,1}_p{}^*X\times\End(E_p)$. In our particular case, this is equivalent to
    \begin{equation}\label{eq:Zpos_dHYM}
        3\I\Tr\Big[\xi^*\wedge\xi\wedge\curvform(h_{FS})+\xi^*\wedge\curvform(h_{FS})\wedge\xi\Big]>\omega_{FS}\wedge\I\Tr\left[\xi^*\wedge\xi\right].
    \end{equation}
    It is sufficient to check this at a single point, as the action of the unitary automorphisms is transitive. So we can perform the computation at the point~$0\in U_0\subset\PP^2$. At this point~$2\pi\omega_{FS}$ is the canonical symplectic form, and the Fubini-Study curvature is (up to a multiple of~$2\pi$)
    \begin{equation}
        \curvform(h_{FS})=\begin{pmatrix}2&0\\0&1\end{pmatrix}\I\dd z^1\wedge\dd\bar{z}^1
        +\begin{pmatrix}0&1\\0&0\end{pmatrix}\I\dd z^1\wedge\dd\bar{z}^2
        +\begin{pmatrix}0&0\\1&0\end{pmatrix}\I\dd z^2\wedge\dd\bar{z}^1
        +\begin{pmatrix}1&0\\0&2\end{pmatrix}\I\dd z^2\wedge\dd\bar{z}^2.
    \end{equation}
    If we let~$\xi=U\dd\bar{z}^1+V\dd\bar{z}^2$ for two matrices~$U$ and~$V$, we find
    \begin{equation}
    \begin{split}
        &\I\Tr\Big[\xi^*\wedge\xi\wedge\curvform(h_{FS})+\xi^*\wedge\curvform(h_{FS})\wedge\xi\Big]=\\
        =&\Big[3\left(\abs{u^1_1}^2+\abs{u^1_2}^2+\abs{v^1_1}^2+\abs{v^2_1}^2\right)+\abs{u^2_1}^2+\abs{u^2_2}^2+\abs{v^1_2}^2+\abs{v^2_2}^2+\\
        &+\abs{u^1_1-v^1_2}^2+\abs{u^2_1-v^2_2}^2+\abs{v^1_1-u^2_1}^2+\abs{v^1_2-u^2_2}^2\Big]\omega_{FS}^2
    \end{split}
    \end{equation}
    while for~$\Tr\left[\xi^*\wedge\xi\right]$ we have
    \begin{equation}
        \omega_{FS}\wedge\I\Tr\left[\xi^*\wedge\xi\right]=\left(\abs{u^1_1}^2+\abs{u^1_2}^2+\abs{u^2_1}^2+\abs{u^2_2}^2+\abs{v^1_1}^2+\abs{v^1_2}^2+\abs{v^2_1}^2+\abs{v^2_2}^2\right)\omega_{FS}^2
    \end{equation}
    and clearly~\eqref{eq:Zpos_dHYM} is satisfied.
\end{exm}

\begin{exm}\label{exm:dHYM_unstable_rank3}
    We consider now a slight modification of~\cite[Example~$2.20$]{DervanMcCarthySektnan}, to see what our conjectures predict for a rank~$3$ bundle over~$X=\PP^2$. Take a Mumford stable bundle~$S\to X$ of rank~$2$ such that there exists a nonzero~$\tau\in H^1(X,S)$. This class~$\tau$ defines a non-split extension of~$\sheaf_{X}$ by~$S$ that we denote by
    \begin{equation}
        0\to S\to E\to \sheaf_{X}\to 0.
    \end{equation}
    It is easy to check that, in this situation,~$E$ and~$S$ satisfy~$\Chern(S)+1=\Chern(E)$ and the Bogomolov inequality implies~$4\Chern_2(S)\leq\Chern_1(S)^2$.
    
    As in~\cite{DervanMcCarthySektnan}, we choose the Fubini-Study form on~$\PP^2$,~$\omega\in H:=\chern_1\left(\sheaf(1)\right)$ and we consider the deformed Hermitian Yang-Mills charge for some B-field class~$B\in H^{1,1}(X,\RR)$, obtained by the choice of weights~$\rho=(-\I,-1,\I/2)$ and unitary class~$U=1-B+B^2/2$,
    \begin{equation}
        \Zch_X^{\mrm{dHYM}}(E)=-\I\int_X\e^{-\I H}\e^{-B}\Chern(E).
    \end{equation}
    More explicitly, the dHYM charge of a bundle $V$
    \begin{equation}
            \Zch_X^{\mrm{dHYM}}(V)
            =\rk(V)\left(H.B-\mu_H(V)+\frac{\I}{2}(1-B^2)+\I\,\mu_B(V)-\I\mu_{MA}(V)\right)
    \end{equation}
    where $\mu_H$ and $\mu_B$ denote the slopes with respect to $H$ and $B$ respectively, while $\mu_{MA}$ is the Monge-Ampère slope.
    
    We proceed to examine the~$\Zch$-stability and~$\Zch$-positivity of~$E$. For the stability part it is sufficient to check the ratio of~$Z_X(S)$ and~$Z_X(E)$. As $S$ and $E$ have the same Chern classes, so that the slopes satisfy $\mu_\bullet(S)=\frac{3}{2}\mu_\bullet(E)$, this ratio is a positive multiple of
    \begin{equation}\label{eq:dHYM_rank3_stability}
    \begin{split}
        \Im\left(\Zch_X(S)\overline{\Zch_X(E)}\right)
        =&\frac{1}{2}(1-B^2)\mu_H(E)+H.B\left(\mu_B(E)-\mu_{MA}(E)\right).
    \end{split}
    \end{equation}
    The case considered in~\cite[Example~$2.20$]{DervanMcCarthySektnan} is~$\deg(S)=0$; in this case, $E$ is~$\Zch$-stable when~$\Chern_2(E)\not=0$ and~$H.B<0$ (note that $\Chern_2(E)\leq 0$ by the Bogomolov inequality). Under these assumptions, Conjecture~\ref{conj:Zstable} predicts that if there exists a~$\Zch$-positive Hermitian metric on~$E$ there should be a~$\Zch$-positive solution of the~$\Zch$-critical equation. If~$H.B\geq 0$ instead we can not apply Theorem~\ref{thm:Zcrit_stability} directly to deduce that there is no dHYM-positive solution of the dHYM equation, as that result applies only to rank~$2$ bundles.

    Theorem~\ref{thm:Zcrit_positivity} suggests that the existence of a dHYM-positive metric should be governed by~$\Zch$-positivity of~$E$ over quotients, this is the content of Conjecture~\ref{conj:Zpositive}. We only consider the~$\Zch$-positivity of~$E$, and the~$\Zch$-positivity for its quotient~$\sheaf_X$; it will be sufficient to consider as subvariety of~$X$ the hyperplane~$H$. The direct computation, using~\eqref{eq:charge_subvariety}, gives
    \begin{equation}
    \begin{split}
        \Zch_H(E_{\restriction H})=&-3\left(\I\,\mu_H(E)-\I\,B.H+1\right)\\
        \Zch_H(\sheaf_X)=&-\left(-\I\,B.H+1\right)=\frac{1}{3}\Zch_H(E_{\restriction H})+\I\,\mu_H(E).
    \end{split}
    \end{equation}
    We are interested in the positivity of the two quantities
    \begin{equation}
        \begin{split}
            \Im\left(\Zch_H(E_{\restriction H})\overline{\Zch_X(E)}\right)=&
            9\left(\left(B.H-\mu_H(E)\right)^2+\frac{1}{2}(1-B^2)+\mu_B(E)-\mu_{MA}(E)\right)\\
            \Im\left(\Zch_H(\sheaf_X)\overline{\Zch_X(E)}\right)=&3\left(B.H(B.H-\mu_H(E))+\frac{1}{2}(1-B^2)+\mu_B(E)-\mu_{MA}(E)\right).
        \end{split}
    \end{equation}
    In particular when~$S$ has degree~$0$ the~$\Zch$-positivity of~$E$ over the quotient~$\sheaf_X$ does not impose any additional condition. Recall however that positivity over quotients is only relevant under condition~\eqref{eq:Z_alphapositive}, i.e.~$\alpha>0$. For our choice of charge, and assuming~$\deg(S)=0$, this is equivalent to
    \begin{equation}
        0<\Im\left(\overline{\Zch_X(E)}\rho_0\right)=-\Re\left(\Zch_X(E)\right)=-3\left(B.H-\mu_H(E)\right)=-3\,B.H
    \end{equation}
    so in this case~\eqref{eq:Z_alphapositive} is in fact equivalent to~$\Zch$-stability. To sum up, let~$B=xH$ for some real number~$x$ and assume that~$\deg(S)=0$. Then we see that~$E$ is~$\Zch$-positive (and positive over the quotient~$\sheaf_X$) if and only if~$1+x^2>\frac{2}{3}\Chern_2(E)$, which is always satisfied. So, Conjecture~\ref{conj:Zpositive} suggests the existence of a~$\Zch$-positive metric on~$E$, if~$x<0$, even though we stated it only for rank~$2$ bundles. Note however that in principle one should also consider all possible quotients of~$E$ restricted to any projective curve.
\end{exm}

\begin{exm}
    We show that there is a rank $3$ bundle over $\PP^2$ that is $\Zch^{dHYM}$-positive, satisfies \eqref{eq:Z_alphapositive}, and is not $\Zch^{dHYM}$-positive over a rank $1$ quotient. This bundle can not admit any $\Zch^{dHYM}$-positive metric by Lemma~\ref{lemma:Zpositivity_rank1quotient}.

    We consider the same setting of Example \ref{exm:dHYM_unstable_rank3}, but assuming that the $B$-field vanishes, for simplicity. The condition \eqref{eq:Z_alphapositive} then becomes $\mu_H(S)>0$, while the $\Zch$-positivity of $E$ and the $\Zch$-positivity for the quotient $\sheaf_X$ are, respectively,
    \begin{equation}
        \begin{split}
            \Im\left(\Zch_H(E_{\restriction H})\overline{\Zch_X(E)}\right)=&4\,\mu_H(S)^2+\frac{9}{2}-6\,\mu_{MA}(S)>0,\\
            \Im\left(\Zch_H(\sheaf_X)\overline{\Zch_X(E)}\right)=&\frac{3}{2}-2\,\mu_{MA}(S)>0.
        \end{split}
    \end{equation}
    We express these conditions in terms of the Chern classes $\chern_1$ and $\chern_2$ of $S$, keeping in mind the Bogomolov inequality $\chern_1^2<4\chern_2$. The condition \eqref{eq:Z_alphapositive} is equivalent to $\chern_1>0$, while for the other two we find
    \begin{equation}
        \begin{dcases}
            \chern_1^2+3\left(\frac{3}{2}-\frac{1}{2}\chern_1^2+\chern_2\right)>0 & \Zch\mbox{-positivity of }E\\
            \frac{3}{2}-\frac{1}{2}\chern_1^2+\chern_2>0 & \Zch\mbox{-positivity over }\sheaf_X.
        \end{dcases}
    \end{equation}
    It is clear that $\Zch$-positivity over the quotient $\sheaf_X$ implies $\Zch$-positivity of $E$, but for some values of $\chern_1$ and $\chern_2$ the bundle will be $\Zch$-positive but not $\Zch$-positive over the quotient. This happens exactly when $\chern_1>0$ and
    \begin{equation}\label{eq:positive_nonpositive}
        3+2\chern_2\leq\chern_1^2<4\chern_2.
    \end{equation}
    For many choices of positive integers $c_1$ and $c_2$ satisfying \eqref{eq:positive_nonpositive} there is a stable bundle $S$ on $\PP^2$ with those Chern numbers, see \cite{DrezetLePotier_stable_P2}. \todo[inline]{An example is $S=T\PP^2$: in this case the Chern numbers are $\chern_1=\chern_2=3$, but $H^1(T\PP^2)=0$. So $E$ will be a decomposable bundle. Which is still ok, it does give an example, but it would be better to have a non-decomposable one.}
\end{exm}

\begin{exm}\label{ex:dHYM_rank3_extension}
    As a particular case of the construction in Example~\ref{exm:dHYM_unstable_rank3}, we fix a $B$-field class $B=xH$ represented by $x\,\omega$ for a real number $x$, and we take~$S=T\PP^2\otimes K_{\PP^2}$. Of course~$S$ is stable, as~$T\PP^2$ is stable. Moreover,~$H^1(X,S)\cong H^{1,1}(X)^\vee$ by Serre duality, so we can define a nontrivial extension~$E$ by using the Fubini-Study form.

    The components of the Chern character are
    \begin{equation}\label{eq:dHYM_rank3_example_classes}
        \begin{split}
            \chern_1(E)=&\chern_1(T\PP^2)+2\chern_1(K_{\PP^2})=-3H\\
            \Chern_2(E)=&\chern_1(K_{\PP^2})^2+\chern_1(T\PP^2).\chern_1(K_{\PP^2})+\Chern_2(T\PP^2)=\frac{3}{2}
        \end{split}
    \end{equation}
    and considering again the dHYM charge~$\Zch=\Zch^{dHYM}$ of Example~\ref{exm:dHYM_unstable_rank3}, we find that the~$\Zch$-positivity,~$\Zch$-positivity over the quotient~$\sheaf_X$, and~$\Zch$-stability of~$E$ become respectively
    \begin{equation}
        \begin{cases}
            \left(x+1\right)^2+2>1;\\
            x(x+1)+\frac{1}{2}(1-x^2)-x-\frac{1}{2}>0;\\
            -\left(x^2+1\right)-x<0.
        \end{cases}
    \end{equation}
    Hence~$E$ is~$\Zch$-positive,~$\Zch$-positive over~$\sheaf_X$, and~$\Zch$-stable for every choice of~$x$, and there should be a~$\Zch$-positive solution of the dHYM equation. We will compute it explicitly.

    The dHYM charge of~$E$ as a function of~$x$ is
    \begin{equation}
        \Zch_X(E)=3\,x+3-\frac{3}{2}\I\left(x+2\right)x.
    \end{equation}
    To write the dHYM equation we choose~$x\,\omega$ as a representative of the class~$B=x\,H$, and the coefficients of the equation are determined by
    \begin{equation}
    \begin{dcases}
        \abs{\Zch_X(E)}\,\alpha=\frac{1}{2}\Im\left(\overline{Z_X(E)}(-\I)\right)=-\frac{3}{2}(1+x)\\
        \abs{\Zch_X(E)}\,\beta=\Im\left(\overline{Z_X(E)}\left(\I\,x-1\right)\right)\omega=\frac{3}{2}x^2\omega\\
        \abs{\Zch_X(E)}\,\gamma=\frac{1}{2}\Im\left(\overline{Z_X(E)}(\I (1-x^2)+2x)\right)\omega^2=\frac{3}{2}\left(1+x+x^2\right)\omega^2.
    \end{dcases}
    \end{equation}
    Hence, the dHYM equation is equivalent to
    \begin{equation}\label{eq:dHYM_rank3}
        -(1+x)\curvform(h)^2+x^2\omega\wedge\curvform(h)+\left(1+x+x^2\right)\omega^2\otimes\id_E=0
    \end{equation}
    and the volume form hypothesis is always satisfied.

    Take~$h=h_S+h_{\sheaf_{\PP^2}}$ where~$h_{\sheaf_{\PP^2}}$ is the flat metric on~$\sheaf_{\PP^2}$, and~$h_S$ is the product of the Fubini-Study Hermitian metrics on~$T{\PP^2}$ and~$K_{\PP^2}$. Then, letting~$A$ be the second fundamental form of~$S\subset E$,
    \begin{equation}
    \curvform(h)=\begin{pmatrix}
        \curvform_S-\frac{\I}{2\pi}A\wedge A^* & \frac{\I}{2\pi}D'A \\ -\frac{\I}{2\pi}D''A^* & -\frac{\I}{2\pi}A^*\wedge A
    \end{pmatrix}
    \end{equation}
    and~$\curvform_S$ satisfies
    \begin{equation}\label{eq:curvature_S}
        \curvform_{S}=\curvform_{T\PP^2}\otimes\id_{K_X}+\id_{T\PP^2}\otimes\curvform_{K_X}=\curvform_{T\PP^2}\otimes\id_{K_X}-3\,\omega\otimes\id_S.
    \end{equation}
    The dHYM equation~\eqref{eq:dHYM_rank3} can be seen as a system of equations for
    \begin{equation}
        A\in\Alt^{0,1}(X,\Hom(\sheaf_X,S))=\Alt^{0,1}(X,S).
    \end{equation}
    Explicitly, in the usual local coordinate system over~$U_0\subset\PP^2$, we take a local trivialisation~$S$ by the local frame~$\partial_{z^b}\otimes\dd z^c\wedge\dd z^d$, and the second fundamental form~$A$ can be written as
    \begin{equation}
        A=\left(A\indices{_{\bar{a}}^b_c_d}\dd\bar{z}^a\right)\otimes\partial_{z^b}\otimes\dd z^c\wedge\dd z^d
    \end{equation}
    where we added the parenthesis to highlight the distinction between the form and the bundle parts. We choose the coefficients of~$A$ as the unique solutions of~$A\indices{_{\bar{a}}^b_b_c}=-(g_{FS})_{c\bar{a}}$. To ease the notation, we denote by~$e_a=\partial_{z^a}\otimes\dd z^1\wedge\dd z^2$ the local frame for~$S$, with dual co-frame~$\varepsilon^a=\dd z^a\otimes\partial_{z^1}\wedge\partial_{z^2}$. We also let~$r^2:=\abs{z^1}^2+\abs{z^2}^2$, so that
    \begin{equation}\label{eq:second_fundform_dHYM}
        A=\frac{1}{(1+r^2)^2}\left[\Big(z^1\bar{z}^2\dd\bar{z}^1-(1+\abs{z^1}^2)\dd\bar{z}^2\Big)\otimes e_1-\Big(\bar{z}^1z^2\dd\bar{z}^2-(1+\abs{z^2}^2)\dd\bar{z}^1\Big)\otimes e_2\right].
    \end{equation}
    While for its adjoint~$A^*$ we find
    \begin{equation}
        A^*=\left(\overline{A\indices{_{\bar{a}}^b_{12}}}(g_{FS})_{c\bar{b}}\det(g_{FS})^{-1}\dd z^a\right)\otimes\varepsilon^c=(1+r^2)\left[\dd z^1\otimes\varepsilon^2-\dd z^2\otimes\varepsilon^1\right].
    \end{equation}
    We claim that this choice of~$A$ satisfies~\eqref{eq:dHYM_rank3}, and that in fact the connection defined by~$E$ is projectively flat, satisfying
    \begin{equation}\label{eq:ex_projectively_flat}
        \curvform(h)=-\omega\otimes\id_E.
    \end{equation}    
    As~$\curvform(h)$ is equivariant with respect to the unitary action on~$\PP^2$ and its lift to~$E$, it will be sufficient to check~\eqref{eq:ex_projectively_flat} at the point~$p=\{z^1=z^2=0\}$, for which we have
    \begin{equation}
        \begin{split}
            A=&-\dd\bar{z}^2\otimes e_1+\dd\bar{z}^1\otimes e_2,\\
            A^*=&-\dd z^2\otimes\varepsilon^1+\dd z^1\otimes\varepsilon^2.
        \end{split}
    \end{equation}
    We claim that~$D'A$ vanishes at~$p$. To see this, write~$D'A$ in the same local frame for~$S$:
    \begin{equation}
        D'A=\del\left(A\indices{_{\bar{a}}^i}\dd\bar{z}^a\right)\otimes e_i+\left(A\indices{_{\bar{a}}^i}\dd\bar{z}^a\right)\wedge D'e_i.
    \end{equation}
    Note first that, at the point~$p$,~$\del(A\indices{_{\bar{a}}^i}\dd\bar{z}^a)=0$ for all indices~$a,i$, from~\eqref{eq:second_fundform_dHYM}.
    So it remains to show that~$(D'e_i)_{\vert p}=0$. But~$e_i=\partial_{z^i}\otimes\dd z^1\wedge\dd z^2$, and~$D'$ is the~$(1,0)$-part of the Chern connection of the Fubini-Study metric on~$T\PP^2\otimes K_{\PP^2}$, so~$D'e_i$ is a sum of Christoffel symbols of the Fubini-Study metric, and they all vanish at~$p$.

    This observation already shows that the off-diagonal components of~$\curvform(h)^2$ and~$\curvform(h)\wedge\omega$ with respect to the decomposition~$E=S+\sheaf_X$ vanish. We proceed to compute the other components of~$\curvform(h)$. The direct computation gives, at the point~$p$
    \begin{equation}
        \begin{gathered}
        \frac{\I}{2\pi}A^*\wedge A=\frac{1}{2\pi}\left(\I\dd z^2\wedge\dd\bar{z}^2+\I\dd z^1\wedge\dd\bar{z}^1\right)=\omega;\\
        A\wedge A^*=
            \begin{array}{l}
                -\dd z^2\wedge\dd\bar{z}^2\otimes e_1\otimes\varepsilon^1+\dd z^2\wedge\dd\bar{z}^1\otimes e_2\otimes\varepsilon^1\\
                +\dd z^1\wedge\dd\bar{z}^2\otimes e_1\otimes\varepsilon^2-\dd z^1\wedge\dd\bar{z}^1\otimes e_2\otimes\varepsilon^2
            \end{array}
        =\begin{pmatrix}
            -\dd z^2\wedge\dd\bar{z}^2 & \dd z^1\wedge\dd\bar{z}^2\\
            \dd z^2\wedge\dd\bar{z}^1 & -\dd z^1\wedge\dd\bar{z}^1
        \end{pmatrix}.
        \end{gathered}
    \end{equation}
    We also know from Example~\ref{ex:not_subvarstable}
    \begin{equation}
        \curvform_{T\PP^2}\otimes\id_{K_X}=\frac{1}{2\pi}
        \begin{pmatrix} 2\I\dd z^1\wedge\dd\bar{z}^1+\I\dd z^2\wedge\dd\bar{z}^2 & \I\dd z^1\wedge\dd\bar{z}^2 \\ \I\dd z^2\wedge\dd\bar{z}^1 & \I\dd z^1\wedge\dd\bar{z}^1+2\I\dd z^2\wedge\dd\bar{z}^2 \end{pmatrix}.
    \end{equation}
    So~$\curvform_{T\PP^2}\otimes\id_{K_X}-\frac{\I}{2\pi}A\wedge A^*=2\omega\otimes\id_S$ and from~\eqref{eq:curvature_S} we finally obtain~\eqref{eq:ex_projectively_flat} at~$p$. This shows that the connection on~$E$ defined by~\eqref{eq:second_fundform_dHYM} is a solution of the dHYM equation, and in fact it will solve \emph{any}~$\Zch$-critical equation. We would like to show that it is also dHYM-positive, and again it is sufficient to verify this at the same point~$p$, for which the computations are straightforward. We need to consider the endomorphism
    \begin{equation}
        R:=2\alpha\,\curvform+\beta\otimes\id_E=\abs{\Zch_X(E)}\frac{3}{2}\left(1+(1+x)^2\right)\omega\otimes\id_E,
    \end{equation}
    which clearly satisfies~$\I\Tr\left[ \xi^*\wedge\xi\wedge R+\xi^*\wedge R\wedge \xi\right]>0$ for every~$\xi\in\Alt^{0,1}(\End(E_p))$.
\end{exm}

The bundle in Example~\ref{ex:dHYM_rank3_extension} thus admits a metric that is~$\Zch$-critical for any choice of polynomial central charge. However, it is \emph{not}~$\Zch$-stable for every charge, due to possible failures of~$\Zch$-positivity. Using the stability result of Corollary \ref{cor:rank2_MAstability} we can exhibit a central charge such that~$E$ admits a~$\Zch$-critical metric that \emph{can not} be~$\Zch$-positive. see Remark \ref{rmk:corank_1_stability} for the precise statement we need.

\begin{exm}\label{ex:not_positive_P2}
    Consider a polynomial central charge~$\Zch$ given by a vector~$\rho=(\rho_0,\rho_1,\rho_2)$ and the short exact sequence over~$X=\PP^2$ of Example~\ref{ex:dHYM_rank3_extension}. For simplicity, we assume that the unitary class~$U$ is trivial,~$U=1$. Then the charges of~$E$ and~$S$ are, using~\eqref{eq:dHYM_rank3_example_classes},
    \begin{equation}
    \begin{split}
        \Zch_X(E)=&3\,\rho_2-3\,\rho_1+\frac{3}{2}\rho_0\\
        \Zch_X(S)=&2\,\rho_2-3\,\rho_1+\frac{3}{2}\rho_0=\Zch_X(E)-\rho_2.
    \end{split}
    \end{equation}
    To apply Corollary~\ref{cor:rank2_MAstability} we need~\eqref{eq:Z_alphapositive}, i.e.~$\alpha>0$. Up to a multiple,~$\alpha$ is
    \begin{equation}
        \Im\left(\overline{\Zch_X(E)}\rho_0\right)=3\,\Im\left(\bar{\rho}_2\rho_0\right)-3\,\Im\left(\bar{\rho}_1\rho_0\right)
    \end{equation}
    and so we have~\eqref{eq:Z_alphapositive} if and only if~$\Im\left(\bar{\rho}_2\rho_0\right)>\Im\left(\bar{\rho}_1\rho_0\right)$. Instead, to check~$\Zch$-stability with respect to the sub-bundle~$S$ we first compute
    \begin{equation}
        \begin{split}
            \Im\left(\overline{\Zch_X(E)}\Zch_X(S)\right)=&\Im\left(\bar{\rho}_2\,\Zch_X(S)\right)=\Im\left(\bar{\rho}_2\left(2\,\rho_2-3\,\rho_1+\frac{3}{2}\rho_0\right)\right)\\  =&-3\,\Im(\bar{\rho}_2\,\rho_1)+\frac{3}{2}\,\Im(\bar{\rho}_2\,\rho_0)
        \end{split}
    \end{equation}
    and we conclude that if~$E$ is~$\Zch$-stable, then it must be
    \begin{equation}
        \Im(\bar{\rho}_2\,\rho_0)<2\,\Im(\bar{\rho}_2\,\rho_1).
    \end{equation}
    By Corollary~\ref{cor:rank2_MAstability} we deduce that~$E$ can not admit~$\Zch$-positive and~$\Zch$-critical metrics if
    \begin{equation}
        \begin{cases}
            \Im\left(\bar{\rho}_2\rho_0\right)>\Im\left(\bar{\rho}_1\rho_0\right)\\
            \Im(\bar{\rho}_2\,\rho_0)>2\,\Im(\bar{\rho}_2\,\rho_1).
        \end{cases}
    \end{equation}
    Many charges satisfy this, an example is~$\rho=(1,-\I,-1-3\I)$. On the other hand, we know that~$E$ admits a connection that is~$\Zch^\rho$-critical, so this connection can not be~$\Zch^\rho$-positive. In fact, with this choice of~$\rho$ the bundle is not even~$\Zch^\rho$-positive, and it can be easily checked that if this bundle is~$\Zch^\rho$-positive, it must be~$\Zch^\rho$-stable.
\end{exm}

\begin{exm}
    Let~$E$ be a vector bundle on a K\"ahler surface~$X$. Given a proper saturated subsheaf~$S\subset E$, we can construct an infinite family of polynomial central charges~$\Zch$ such that~$\Im\left(\Zch_X(S)\overline{\Zch_X(E)}\right)>0$, i.e. a charge for which~$E$ is~$\Zch$-unstable.

    We can in fact achieve this just by an appropriate choice of the unitary class. The computation is very similar to those in Proposition~\ref{prop:stability_comparison}, but we rewrite them to highlight the dependence on the unitary class. Denote by~$H$ the class of the K\"ahler form, assume that~$\int_XH^2=1$, and choose~$U=1+xH+yH^2$ for two real numbers~$x,y$ that will be fixed later. Then the charges of~$S$ and~$E$ respectively are (c.f. proof of Proposition~\ref{prop:stability_comparison})
    \begin{equation}
        \begin{split}
            \Zch_X(E)=\rho_0\rk(E)\left(y+x\mu_H(E)+\mu_{MA}(E)\right)+\rho_1\rk(E)\left(\mu_H(E)+x\right)+\rho_2\rk(E)\\
            \Zch_X(S)=\rho_0\rk(S)\left(y+x\mu_H(S)+\mu_{MA}(S)\right)+\rho_1\rk(S)\left(\mu_H(S)+x\right)+\rho_2\rk(S).
        \end{split}
    \end{equation}
    Hence we obtain
    \begin{equation}\label{eq:Zstab_subsheaf_general}
    \begin{split}
        \frac{\Im\left(\Zch_X(S)\overline{\Zch_X(E)}\right)}{\rk(E)\rk(S)}=&
        y\,\Im\left(\rho_0\bar{\rho}_1\right)\left(\mu_H(E)-\mu_H(S)\right)+x^2\Im\left(\rho_0\bar{\rho}_1\right)\left(\mu_H(S)-\mu_H(E)\right)\\
        +&x\big[\Im\left(\rho_0\bar{\rho}_1\right)\left(\mu_{MA}(S)-\mu_{MA}(E)\right)+\Im\left(\rho_0\bar{\rho}_2\right)\left(\mu_H(S)-\mu_H(E)\right)\big]\\
        +&\Im\left(\rho_0\bar{\rho}_1\right)
        \left(\mu_{MA}(S)\mu_H(E)-\mu_{MA}(E)\mu_H(S)\right)\\
        +&\Im\left(\rho_1\bar{\rho}_2\right)\left(\mu_H(S)-\mu_H(E)\right)+\Im\left(\rho_0\bar{\rho}_2\right)\left(\mu_{MA}(S)-\mu_{MA}(E)\right).
    \end{split}
    \end{equation}
    This is a polynomial expression in~$x$ and~$y$ of the form
    \begin{equation}
        a\,y-a\,x^2+b\,x+c.
    \end{equation}
    If the coefficient of~$y$ is nonzero, i.e. if~$\mu_H(S)\not=\mu_H(E)$, then we can fix any~$x$, say~$x=0$, and choose~$y$ so that~\eqref{eq:Zstab_subsheaf_general} is positive, so that~$S$ will destabilise~$E$. In particular, if~$E$ is Mumford stable the coefficient of~$y$ in~\eqref{eq:Zstab_subsheaf_general} is positive, and if we choose~$y>0$ very large the right-hand side of~\eqref{eq:Zstab_subsheaf_general} will be positive.
    
    Assuming instead~$\mu_H(S)=\mu_H(E)$, the coefficient of~$x^2$ in~\eqref{eq:Zstab_subsheaf_general} vanishes, and the coefficient of~$x$ is~$\Im\left(\rho_0\bar{\rho}_1\right)\left(\mu_{MA}(S)-\mu_{MA}(E)\right)$. If~$\mu_{MA}(S)\not=\mu_{MA}(E)$ we can then choose~$x$ such that~$S$ is a destabilising subsheaf with respect to the~$\Zch$-stability, while if the two Monge-Ampère slopes are equal (as are the Mumford slopes) then~\eqref{eq:Zstab_subsheaf_general} vanishes, and again this means that~$E$ is~$\Zch$-unstable.

    We can do a similar analysis for the~$\Zch$-positivity of the bundle or of a Hermitian metric~$h\in\Hermmetric(E)$, assuming that we choose~$1+x\,\omega+y\,\omega^2$ as representative of the unitary class. Then we compute
    \begin{equation}
    \begin{split}
        2\alpha\,\Tr\curvform(h)+\rk(E)\beta=&\Im\left(\frac{\overline{\Zch_X(E)}}{\abs{\Zch_X(E)}}\rho_0\right)\left(2\Tr\curvform(h)+x\,\rk(E)\,\omega\right)\\
        &+\rk(E)\Im\left(\frac{\overline{\Zch_X(E)}}{\abs{\Zch_X(E)}}\rho_1\right)\omega.
    \end{split}
    \end{equation}
    The first term on the right-hand side does not depend on~$y$. For the second term, we get
    \begin{equation}
        \frac{1}{\rk(E)}\Im\left(\overline{\Zch_X(E)}\rho_1\right)\omega=
        \Im\left(\overline{\rho_0}\rho_1\right)\left(y+x\mu_H(E)+\mu_{MA}(E)\right)\omega+\Im\left(\overline{\rho_2}\rho_1\right)\omega.
    \end{equation}
    For~$y\gg 0$ this becomes very negative, hence we can choose the charge so that the fixed metric on~$E$ is~$\Zch$-negative.
    
    The same method shows that if~$E$ is a strictly unstable Mumford bundle, we can take~$x=0$ and~$y\ll 0$ so that~$E$ is~$\Zch$-unstable and~$Z$-positive. If~$E$ is strictly Mumford semistable with~$\mu_H(S)=\mu(E)$, after fixing~$x$, for~$y\ll 0$,~$E$ is~$\Zch$-unstable and~$Z$-positive.
\end{exm}

\subsection{Z-polystability}

Consider the case of a decomposable rank~$2$ bundle~$E=L_1\oplus L_2\to X$. If~$h_i\in\Hermmetric(L_i)$ for~$i=1,2$, then we get a metric~$h=h_1\oplus h_2$ on~$E$, and
\begin{equation}
    \curvform(h)=\begin{pmatrix}
        \curvform(h_1) & 0\\ 0 & \curvform(h_2)
    \end{pmatrix}
\end{equation}
under the decomposition~$\End(E)=\End(L_1)+\Hom(L_1,L_2)+\Hom(L_2,L_1)+\End(L_2)$. Hence, for this metric,~\eqref{eq:Zcrit_vbMA} splits as the system
\begin{equation}\label{eq:Zcrit_split}
    \begin{dcases}
        \left(2\alpha\curvform(h_1)+\beta\right)^2=\beta^2-4\alpha\gamma\\
        \left(2\alpha\curvform(h_2)+\beta\right)^2=\beta^2-4\alpha\gamma.
    \end{dcases}
\end{equation}
Under the volume form hypothesis~$\beta^2-4\alpha\gamma>0$, by Yau's solution of the Calabi conjecture, the system~\eqref{eq:Zcrit_split} has solutions under the following conditions:
\begin{equation}\label{eq:split_existenceconditions}
    \begin{dcases}
        2\alpha L_i+[\beta]\mbox{ has a sign} &\mbox{for }i=1,2\\
        (2\alpha L_i+[\beta])^2=[\beta^2-4\alpha\gamma] &\mbox{for }i=1,2.
    \end{dcases}
\end{equation}
As the bundle~$E$ is decomposable, by Lemma~\ref{lemma:stable_simple} it is not~$\Zch$-stable. However, we have
\begin{lemma}\label{lemma:polystable_rank2}
    The following are equivalent:
    \begin{enumerate}
        \item~$\Im\left(\frac{\Zch_X(S)}{\Zch_X(E)}\right)\leq 0$ for any proper sub-bundle~$S\subset E$;
        \item~$\Im\left(\Zch_X(L_1)\overline{\Zch_X(L_2)}\right)=0$;
        \item~$(2\alpha L_1+[\beta])^2=(2\alpha L_2+[\beta])^2=[\beta^2-4\alpha\gamma]$.
    \end{enumerate}
\end{lemma}
\begin{proof}
    The equivalence between~$1$ and~$2$ is a direct consequence of the fact that each~$L_i$ is a line bundle, and that the sequence~$0\to L_1\to E\to L_2\to 0$ splits, so that~$\Zch_X(E)=\Zch_X(L_1)+\Zch_X(L_2)$. Note also that each condition is equivalent to~$\Im\left(\overline{\Zch_X(E)}\Zch_X(L_i)\right)=0$ for~$i=1,2$. From Proposition~\ref{prop:stability_comparison} we know that~$\Im\left(\overline{\Zch_X(E)}\Zch_X(L_i)\right)=0$ is equivalent to the equality of Monge-Ampère slopes~$\mu_{MA,\vartheta}(L_i)=\mu_{MA,\vartheta}(E)$. By definition of the Monge-Ampère slope we get, for~$\vartheta:=[\beta/2\alpha]$,
    \begin{equation}
    \begin{split}
        \mu_{MA,\vartheta}(L_i)=&\Chern_2(L_i)+\Chern_1(L_i).\vartheta\\
        \mu_{MA,\vartheta}(E)=&\frac{\Chern_2(E)}{2}+\frac{\Chern_1(L_1).\vartheta+\Chern_1(L_2).\vartheta}{2}
    \end{split}
    \end{equation}
    hence the Monge-Ampère slopes are equal if and only if
    \begin{equation}
        (2\alpha L_1+[\beta])^2=(2\alpha L_2+[\beta])^2
    \end{equation}
    and they must each equal~$[\beta^2-4\alpha\gamma]$ since~$\Chern_2(E)+\Chern_1(E).\vartheta+\vartheta^2=\vartheta^2-\gamma/\alpha$.
\end{proof}
The conditions in~\eqref{eq:split_existenceconditions} imply the existence of solutions of the~$\Zch$-critical equation, so they must be part of a hypothetical~$\Zch$-polystability condition on~$E=L_1\oplus L_2$. From (the proof of)~\cite[Theorem~$1.4$]{DervanMcCarthySektnan} we know that, \emph{if~$\beta^2-4\alpha\gamma$ is a volume form}, then the following conditions are equivalent for each~$i=1,2$:
\begin{enumerate}
    \item~$L_i$ admits a~$\Zch$-critical metric;
    \item~$L_i$ admits a~$\Zch$-positive metric;
    \item~$2\alpha\,L_i+[\beta]>0$.
\end{enumerate}
In particular, under the volume form hypothesis,~\cite[Theorem~$1.4$]{DervanMcCarthySektnan} implies Conjecture~\ref{conj:Zstable} for line bundles over surfaces. Looking at the conditions in~\eqref{eq:split_existenceconditions} and Lemma~\ref{lemma:polystable_rank2}, it is then natural to interpret the equation~$\Im\left(\Zch_X(L_1)\overline{\Zch_X(L_2)}\right)=0$ as part of a hypothetical polystability condition for~$E=L_1\oplus L_2$, which motivates an extension of Definition~\ref{def:ZpositiveZstable} and Conjecture~\ref{conj:Zstable} to the case of decomposable bundles.
\begin{definition}\label{def:polystable}
    Given a polynomial central charge Z, a vector bundle~$E$ is~$\Zch$-\emph{poly}stable if it is~$\Zch$-semistable and it is a direct sum~$E=\bigoplus E_i$ of bundles that are~$\Zch$-stable, so that for every~$i,j$
    \begin{equation}\label{eq:polystab_condition}
        \Im\left(\Zch_X(E_i)\overline{\Zch_X(E_j)}\right)=0.
    \end{equation}
\end{definition}

\begin{conj}\label{conj:ZpositiveZpolystable}
    For any polynomial central charge~$\Zch$ and any holomorphic vector bundle~$E$ on a compact K\"ahler surface, there exists a~$\Zch$-positive solution~$h\in\Hermmetric(E)$ of the~$\Zch$-critical equation if and only if~$E=\bigoplus_i E_i$ is~$\Zch$-polystable, and each of the~$E_i$s is~$\Zch$-positive and~$\Zch$-positive for quotients.
\end{conj}
Again, it is clear that this conjecture can only hold under some additional assumption, such as a hypothetical \emph{supercritical phase condition} that is yet to be understood. For the toy example we are examining, i.e.~$E=L_1\oplus L_2$, we know that we should impose the volume form hypothesis
\begin{equation}
    \beta^2-4\alpha\gamma>0.
\end{equation}
We have already proven that, under the volume form hypothesis, if~$L_1$ and~$L_2$ are~$\Zch$-positive and satisfy~\eqref{eq:polystab_condition}, then there is a~$\Zch$-positive solution of the~$\Zch$-critical equation on~$E$. Our results in Section~\ref{sec:positivity_stability} give us the converse implication as well.
\begin{lemma}
    Let~$\Zch$ be a polynomial central charge over a K\"ahler surface that satisfies the volume form hypothesis. Then Conjecture~\ref{conj:ZpositiveZpolystable} holds for any rank~$2$ decomposable vector bundle~$E=L_1\oplus L_2$.
\end{lemma}
\begin{proof}
    It remains to show that if~$E=L_1\oplus L_2$ has a~$\Zch$-positive solution of the~$\Zch$-critical equation then the two bundles~$L_i$ are themselves~$\Zch$-positive and satisfy~\eqref{eq:polystab_condition}, as~$\Zch$-stability and~$\Zch$-positivity for quotients are void conditions on line bundles. We prove this by rephrasing it through the correspondence with the Monge-Ampère equation of Lemma~\ref{lemma:equation_equivalence}. We consider only the case when~$\alpha>0$, the other situation is completely symmetrical. 
    
    Assume that~$E=L_1\oplus L_2$ has a Monge-Ampère--positive metric~$h\in\Hermmetric(E)$ (with respect to~$\beta$) that solves the twisted vector bundle Monge-Ampère equation~\eqref{eq:vbMA_twisted}. By (the proof of) Corollary~\ref{cor:rank2_MAstability} then~$\mu_{MA,\vartheta}(L_i)\leq\mu_{MA,\vartheta}(E)$ for~$i=1,2$. However~$E=L_1\oplus L_2$ as a holomorphic vector bundle, so~$\mu_{MA,\vartheta}(L_i)=\mu_{MA,\vartheta}(E)$ and this implies that the second fundamental form of~$h$ with respect to the inclusion~$L_1\subset E$ vanishes (c.f. proof of Corollary~\ref{cor:rank2_MAstability}). From~\eqref{eq:curvature_secondfundform} the curvature of~$h$ satisfies
    \begin{equation}\label{eq:curv_decomposition_directsum}
        \curvform_{\alpha,\beta}(h)=\begin{pmatrix}
                    \curvform_{\alpha,\beta}(h_{L_1}) & 0\\
                    0 & \curvform_{\alpha,\beta}(h_{L_2})
                \end{pmatrix}
    \end{equation}
    where~$h_{L_i}$ denotes the restriction of~$h$ to~$L_1$ and~$L_2$ respectively. Hence~$h_{L_i}$ satisfy
    \begin{equation}
        \curvform_{\alpha,\beta}(h_{L_1})^2=\eta=\curvform_{\alpha,\beta}(h_{L_2})^2
    \end{equation}
so~\eqref{eq:polystab_condition} holds. As we are assuming that~$h$ is Monge-Ampère positive, from~\eqref{eq:curv_decomposition_directsum} it also easily follows that each~$h_{L_i}$ must be Monge-Ampère positive.
\end{proof}

\begin{rmk}
Not every pair of solutions~$(h_1,h_2)$ of~\eqref{eq:Zcrit_split} gives a~$\Zch$-critical metric~$h=h_1+h_2$ that is also~$\Zch$-positive, as the signs of~$2\alpha L_i+[\beta]$ might not be the same.
\end{rmk}

\begin{exm}
  Consider a~$E=L_1\oplus L_1^*$ where~$L_1$ is an ample line bundle over a surface. If~$\alpha\neq 0$ and~$L_1$ and~$L_1^*$ are~$\Zch$-positive, this implies that~$2\alpha L_1^*+[\beta]>0$ and~$2\alpha L_1+[\beta]>0$. Consequently,~$[\beta]$ is positive. Now ~$E$ cannnot be~$\Zch$-polystable, since Lemma~\ref{lemma:polystable_rank2} and Equation~\ref{eq:polystab_condition} imply that~$2\alpha L_1[\beta]=0$. If~$\alpha=0$, the bundle is not Mumford polystable since~$L_1$ and~$L_1^*$ have different degrees. Consequently, we have seen that for any~$Z$-charge, such bundle~$E$ is can not be simultanously ~$\Zch$-positive and~$\Zch$-polystable.
\end{exm}

\begin{rmk}\label{rmk:polystab_issue}
    It is natural to ask if a direct sum of~$\Zch$-stable bundles~$E=\bigoplus_i E_i$ that satisfy~\eqref{eq:polystab_condition} is~$\Zch$-semistable. This is indeed the case if the~$E_i$ are line bundles, but it might fail in general. We can assume that~$E$ has just two components,~$E=E_1\oplus E_2$. If~$S\subset E$ is a coherent saturated subsheaf of rank~$0<\rk(S)<\rk(E)$, we have
    \begin{equation}
        \begin{tikzcd}
            0 \ar{r} & E_1 \ar{r} & E_1\oplus E_2 \ar{r} & E_2 \ar{r} & 0\\
            0 \ar{r} & S_1 \ar{r} \ar[hookrightarrow]{u} & S \ar{r} \ar[hookrightarrow]{u} & S_2 \ar{r} \ar[hookrightarrow]{u} & 0
        \end{tikzcd}        
    \end{equation}
    for~$S_1=S\cap E_1\times\{0\}$,~$S_2$ is the projection of~$S$ induced by~$E_1\oplus E_2\to E_2$, and all vertical arrows are injective. To check semistability, we should compute
    \begin{equation}\label{eq:polystab_semistab}
        \Im\left(\Zch_X(S)\overline{\Zch_X(E)}\right) = \sum_{i,j=1}^2\Im\left(\Zch_X(S_i)\overline{\Zch_X(E_j)}\right)
    \end{equation}
    and check that it is always non-positive. If the~$E_i$ are line bundles satisfying~\eqref{eq:polystab_condition}, then~\eqref{eq:polystab_semistab} always vanishes so there is nothing to check. Otherwise,~\eqref{eq:polystab_condition} implies that there exists an angle~$\e^{\I\vartheta}$ and real numbers~$r_1,r_2\in \mathbb{R}^*$ such that~$\Zch_X(E_i)=r_i\e^{\I\vartheta}$ for~$i=1,2$. Then~\eqref{eq:polystab_semistab} becomes
    \begin{equation}
        \Im\left(\Zch_X(S)\overline{\Zch_X(E)}\right) = 
        (r_1+r_2) \left(\frac{\Im\left(\Zch_X(S_1)\overline{\Zch_X(E_1)}\right)}{r_1}+\frac{\Im\left(\Zch_X(S_2)\overline{\Zch_X(E_2)}\right)}{r_2}\right).
    \end{equation}
    As each~$E_i$ is~$\Zch$-stable,~$\Im\left(\Zch_X(S_i)\overline{\Zch_X(E_i)}\right)\leq 0$ for~$i=1,2$, but the above sum might still be positive if~$r_1$ and~$r_2$ do not have the same sign. Some additional conditions on the bundles~$E_i$ however guarantee that~$E=E_1\oplus E_2$ is semistable. For example if the bundles~$E_i$ all satisfy condition~\eqref{eq:Z_alphapositive} (so that~$\alpha_{E_i}>0$, see Remark~\ref{rmk:strong_Zpositive}) then~$r_1$ and~$r_2$ have the same sign.
\end{rmk}

\subsection{Z-stability and Gieseker stability}

In his PhD thesis~\cite{Leung_equation_thesis}, Leung introduced the notion of \emph{almost Hermitian-Einstein} metric. For a given holomorphic vector bundle~$E$ over a polarised compact complex manifold~$L\to X$ and~$k$ sufficiently large, a metric~$h_k$ is said to be \emph{almost Hermitian-Einstein} if its Chern connection satisfies
\begin{equation}\label{eq:almostHE_Leung}
\left[\exp\left(\curvform(h_k)+k\,\omega\otimes\id_E\right)\wedge\Todd_X\right]^\topdeg=c_k\,\frac{\omega^n}{n!}\otimes\id_E
\end{equation}
where~$\omega$ is a K\"ahler metric in~$\chern_1(L)$ and~$\Todd_X$ is the~$\omega$-harmonic representative of the Todd class of~$X$. This equation has been introduced as an analytic counterpart to Gieseker stability.

Note that~$\curvform(h_k)+k\,\omega\otimes\id_E$ is the curvature form of a connection on~$E\otimes L^k$; then, by Hirzebruch-Riemann-Roch we have
\begin{equation}
\int_X\Tr\left[\exp\left(\curvform(h_k)+k\,\omega\otimes\id_E\right)\wedge\Todd_X\right]^\topdeg=\Chern(E\otimes L^k).\Todd_X=\chi(E\otimes L^k),
\end{equation}
where~$\chi(E\otimes L^k)$ denotes the Euler Characteristic of the product bundle, so the constant~$c_k$ in~\eqref{eq:almostHE_Leung} must be
\begin{equation}
c_k=\frac{\chi(E\otimes L^k)}{\vol(X)\,\rk E}.
\end{equation}
Leung's equation~\eqref{eq:almostHE_Leung} can be reinterpreted as a special case of the~$\Zch$-critical equation~\eqref{eq:Zcrit}, given by the \emph{almost Hermite-Einstein} charge
\begin{equation}
\Zdiff^{aHE}_{k}(h)=c_k\,\frac{\omega^n}{n!}\otimes\id_E+\I\big(\exp\left(\curvform(h)+k\,\omega\otimes\id_E\right)\wedge\Todd_X\big)^\topdeg.
\end{equation}
Indeed, by integrating over~$X$ we find
\begin{equation}
\int_X\Tr\left(\Zdiff^{aHE}_{k}(h)\right)=\chi(E\otimes L^k)\left(1+\I\right)
\end{equation}
so that the corresponding equation~\eqref{eq:Zcrit} is equivalent to (at least if~$\chi(E\otimes L^k)\not=0$)
\begin{equation}
\begin{split}
0=&\Im\left((1-\I)\Zdiff^{aHE}_{k}(h)\right)=\\
=&-c_k\,\frac{\omega^n}{n!}\otimes\id_E+\big[\exp\left(\curvform(h)+k\,\omega\otimes\id_E\right)\wedge\Todd_X\big]^\topdeg
\end{split}
\end{equation}
which is precisely~\eqref{eq:almostHE_Leung}. Note that~$\Zdiff_k^{aHE}$ can be expressed in a form closer to~\eqref{eq:Zdiff_operator}:
\begin{equation}
\begin{split}
\Zdiff^{aHE}_{k}(h)=&\left[c_k\,\frac{\omega^n}{n!}\e^{\curvform(h)}\Todd_X+\I\e^{k\,\omega}\e^{\curvform(h)}\Todd_X\right]^\topdeg\\
=&\left[\left(\left(\frac{c_k}{k^n}+\I\right)\frac{(k\omega)^n}{n!}\id_E+\left(\sum_{j=0}^{n-1}\frac{\I}{j!}(k\omega)^j\right)\right)\wedge\e^{\curvform(h)}\wedge\Todd_X\right]^\topdeg.
\end{split}
\end{equation}
The stability vector~$\rho^{aHE}:=(\rho_1,\dots,\rho_n)$ is then
\begin{equation}
\begin{dcases}
\rho_j=\frac{\I}{j!} & \mbox{for }0\leq j\leq n-1\\
\rho_n=\frac{1}{n!}\left(\frac{c_k}{k^n}+\I\right).
\end{dcases}
\end{equation}
An important remark is that the almost Hermite-Einstein charge does not fit precisely in the discussion of~\cite{DervanMcCarthySektnan}, as it is not properly a polynomial central charge:~$\rho_{j}/\rho_{j+1}=j+1$ for~$j<n-1$, hence~$\Im(\rho_{j}/\rho_{j+1})$ is \emph{not} positive for~$j<n-1$. It is however true that~$\Im(\rho_{n-1}/\rho_{n})>0$ for~$k\to\infty$ (by asymptotic Riemann-Roch), which is the assumption needed to carry out the asymptotic analysis of the equation in~\cite{DervanMcCarthySektnan}. Another issue is that the coefficient~$\rho_n$ defining the charge depends on the bundle~$E$ itself, so the charge is not additive over short exact sequences, for example.

Nevertheless, given a bundle~$E$, we can still apply our results to the almost Hermite-Einstein charge, by considering the coefficients~$\rho_j$ as fixed. In this way we obtain a~$\Zch$-critical equation for \emph{any} bundle~$F$, which will coincide with the almost Hermite-Einstein equation on~$F$ only if~$\chi(F\otimes L^k)\rk(E)=\chi(E\otimes L^k)\rk(F)$: explicitly, the \emph{almost Hermite-Einstein charge defined by~$E\to X$},~$\Zch^{aHE}_{E,k}$, is defined by the unitary class~$U=\Todd_X$ and the stability vector
\begin{equation}
    \begin{dcases}
        \rho_j=\frac{\I}{j!} & \mbox{for }0\leq j\leq n-1\\
        \rho_n=\frac{1}{n!}\left(\frac{c_k}{k^n}+\I\right)
    \end{dcases}
\end{equation}
so that for any vector bundle~$F$
\begin{equation}
    \Zch^{aHE}_{E,k}(F)=\chi(E\otimes L^k)\frac{\rk(F)}{\rk(E)}+\I\chi(F\otimes L^k).
\end{equation}
In particular for any coherent subsheaf~${F}\subset E$ we have
\begin{equation}
\begin{split}
    \Im\left(\overline{\Zch^{aHE}_{E,k}(E)}\,\Zch^{aHE}_{E,k}(F)\right)=&\chi(E\otimes L^k)\Im\left(\left(1-\I\right)\left(\chi(E\otimes L^k)\frac{\rk(F)}{\rk(E)}+\I\chi(F\otimes L^k)\right)\right)\\
    =&\chi(E\otimes L^k)\left(\chi(F\otimes L^k)-\chi(E\otimes L^k)\frac{\rk(F)}{\rk(E)}\right).
\end{split}
\end{equation}
Since~$\chi(E\otimes L^k)>0$ for~$k\gg 0$, we obtain that ~$\Zch^{aHE}_{E,k}$-stability of~$E$ coincides with Gieseker stability for sufficiently large~$k$. \\
This remark has two applications. Firstly, all our results however apply to this almost-Hermitian charge for any~$k$, thus giving an analogue of Gieseker-stability in a non-asymptotic case, for rank~$2$ bundles over surfaces.
Secondly, the proof of  Theorem~\ref{thm:asymptotic_stability} works in that case and provides the existence of an almost Hermitian-Einstein metric on a Gieseker stable bundle which is sufficiently smooth, completing the proof of the main result of Leung's thesis. Moreover, the techniques of~\cite{DervanMcCarthySektnan} do not require~$k$ to be a integer, it can be chosen as a real positive number sufficiently large.
\begin{exm}\label{exm:suite}
We continue to investigate Example~\ref{exm:CP2blowup} in order to construct a~$\Zch$-positive~$\Zch$-critical metric on~$E$, thus checking Conjecture~\ref{conj:Zstable} in this particular case. As before, we consider a non-split extension~$$0\to \mathcal{L}_r\to E\to \sheaf_X\to 0$$ over~$X$, the blow-up of~$\PP^2$ at one point, but we fix~$[\omega]=3H-E_1=c_1(X)$. 
Given any stability vector~$\rho=(\rho_0,\rho_1,\rho_2)$, we consider ~$U_1$ proportional to~$c_1(X)$ such that~$\alpha>0$ in~\eqref{eq:abc_coeff}. Since~$\ch_2(\mathcal{L}_r)<\ch_2(\mathcal{O}_X)$, the same reasoning as we did for Monge-Ampère stability shows that~$E$ is Gieseker stable (and thus simple) with respect to the integral class~$[\omega]$. From the exact sequence,~$E$ is also sufficiently smooth as the graded object~$Gr(E)=\mathcal{L}_r\oplus E/\mathcal{L}_r$ is a holomorphic bundle. 
 Consequently, there exists~$k_0(r)>0$ such that for any real~$k>k_0(r)$, there exists a metric~$h_{r,k}$ on~$E$ that solves~\eqref{eq:almostHE_Leung} and this simplifies to
\begin{equation}\label{eq:Lsimlified}\mathcal{F}(h_{r,k})^2+ \left(2k+\frac{1}{2}\right)\omega\wedge\mathcal{F}(h_{r,k}) =c'_{r,k}\omega^2 \otimes \id_E 
\end{equation}
where~$c'_{r,k}$ is a topological constant.
We explain now how we will fix the~$\Zch$-charge. 

As we explained before, we have the flexibility to choose~$U_2$ independently of~$r$ such that~$\vert Z_X(E)\vert\beta$ writes as~$\kappa\omega$ where~$\kappa>0$ is as large as we want. Since~$\vert Z_X(E)\vert\alpha$ is positive and does not depend on~$U_2$, the~$\Zch$-critical equation~\eqref{eq:Zcritical_surface} becomes
\begin{equation}\label{eq:Zcritexample}\mathcal{F}(h)^2+\kappa'\omega \wedge \mathcal{F}(h)+\gamma'\omega^2\otimes\id_E=0
\end{equation}
where~$\kappa'=\frac{\kappa}{\alpha}>0$ is large as we want while~$\gamma'$ is a constant, which is fixed topologically. From above discussion, we can choose~$U_2$ independently of~$r$ such that~$\kappa'>2k_0(r)+\frac{1}{2}$ and we can identify the topological constants~$\gamma'=-c'_{r,\kappa/2-1/4}$. Consequently, the almost Hermitian-Einstein metric~$h_{r,\kappa/2-1/4}$ is solution to the~$\Zch$-critical equation.

To sum up, we have obtained an infinite family of Mumford semistable bundles~$E$ (in the~$r$ parameter) such that for any stability vector~$\rho$ and for an infinite choice of unipotent classes~$1+u_1+u_2$,~$E$ is~$\Zch$-stable and admits a~$\Zch$-critical metric. From Theorem~\ref{thm:asymptotic_stability}, the boundedness of the almost Hermitian-Einstein metrics ensure that 
$$\I\Tr\left[\left(2\alpha\,\curvform(h_{r,\kappa/2-1/4})+\beta\otimes\id_E\right)\wedge\xi^*\wedge\xi\right]_{\sym}>0$$
for any~$\xi\in T^{0,1}_p{}^*X\times\End(E_p)$ at~$p\in X$, since~$\kappa$ can be taken arbitrarily large. This shows that we have obtained a~$\Zch$-positive metric.
\end{exm}


\printbibliography

\bigbreak

\noindent Département de Mathématiques, Université du Québec à Montréal (UQAM), C.P. 8888, Succ. Centre-Ville. Montréal (Québec) H3C 3P8, Canada

\smallbreak

\noindent keller.julien.3@uqam.ca\\
scarpa.carlo@courrier.uqam.ca

%

\end{document}